%% file: harpo_v2.tex
\documentclass[11pt]{article}
\usepackage{ulem}
\usepackage{subcaption}
\usepackage{mathtools}
\usepackage[colorlinks=true, linkcolor=blue, citecolor=red]{hyperref}
\usepackage[utf8]{inputenc}
\usepackage[bbgreekl]{mathbbol}
\usepackage{amsmath,amssymb}
\usepackage{amsthm}
\usepackage{graphicx}
\usepackage{url}
\usepackage[numbers,square]{natbib}
\usepackage{multirow}
\usepackage{tikz}
\usepackage{tikz-cd}
\usetikzlibrary{arrows,matrix}
\usepackage{bbm}
\usepackage{authblk}

\usepackage{bm}  
\usepackage{comment}
\usepackage{enumerate}

\usepackage{array}

\usetikzlibrary{matrix}

\usepackage{stmaryrd}

\usetikzlibrary{calc,patterns,shapes.geometric}
\usetikzlibrary{arrows,shapes,positioning}
\usetikzlibrary{decorations.markings,decorations.pathmorphing,decorations.pathreplacing}
\usetikzlibrary{decorations.text,decorations.shapes}

\include{hvp_macros}

\newtheorem{theorem}{Theorem}[section]
\newtheorem{lemma}{Lemma}[section]
\newtheorem{proposition}[lemma]{Proposition}

\newtheorem{remark}[lemma]{Remark}

\newtheorem{definition}[lemma]{Definition}
\newtheorem{assumption}[lemma]{Assumption}

\usepackage{import}

\providecommand{\keywords}[1]
{
  \small	
  \textbf{Keywords --- } #1
}

\providecommand{\msc}[1]
{
  \small	
  \textbf{MSC (2020) --- } #1
}

\title{Harmonic potentials in the de Rham complex}

\author[1]{Martin Campos Pinto}
\author[1,2]{Julian Owezarek}
\affil[1]{Max-Planck-Institut für Plasmaphysik, Garching, Germany}
\affil[2]{Technische Universität München, Garching, Germany}

\begin{document}

\maketitle

\begin{abstract} 
    Representing vector fields by potentials can be a challenging task in domains with cavities or tunnels, due to the presence of harmonic fields which are both irrotational and solenoidal but may have no scalar or vector potentials.
    For harmonic fields normal to the boundary, which exist in domains with cavities, the standard approach is to construct scalar potentials by solving Laplace's equation with Dirichlet boundary conditions fitted to the closed surfaces surrounding the domain's cavities.
    For harmonic fields tangent to the boundary, which exist in domains with tunnels, a similar method was lacking. In this article we present a construction of vector potentials obtained by solving curl-curl problems with inhomogeneous tangent boundary conditions fitted to closed curves looping around the tunnels. Just as the cavity surfaces represent a basis for the 2-chain homology group, these tunnel curves represent a basis for the 1-chain homology group and the corresponding vector potentials yield a basis for the tangent harmonic fields. 
    In our analysis the linear independence of the harmonic fields is established by considering their fluxes through a collection of reciprocal surfaces.
    These surfaces, whose boundaries lie on the boundary of the domain and which are in intersection duality with the tunnel curves, represent a basis for the relative 2-chain homology group modulo the boundary: their existence in general domains follows from the Poincaré-Lefschetz duality.
    Applied to structure-preserving finite elements with commuting projections and standard compatibility properties on the boundaries, our approach provides an exact geometric parametrization of the discrete harmonic fields in terms of (strong) discrete potentials.
    An interesting by-product is a direct proof that the resulting discrete harmonic spaces have the correct dimensions, which does not rely on uniform stability properties for the commuting projections.
\end{abstract}

\bigskip
\bigskip
\bigskip
\bigskip
\bigskip
\bigskip
\bigskip
\bigskip
\bigskip
\bigskip
\bigskip
\bigskip


\bigskip

\msc{78A30, 14F40, 65N30, 31B20}

\newpage 

\tableofcontents


\section{Introduction} \label{sec:intro}

Using scalar or vector potentials to represent vector fields is a powerful technique in physics and engineering.
In practical computations they ensure that a vector field is either irrotational (curl-free) or solenoidal (divergence-free), which has useful applications in many branches such as astrophysics \cite{silberman_numerical_2019,sarafidou_global_2025}, magnetostatics~\cite{biro_magnetostatic_1996,rapetti_discrete_2003,lipnikov_mimetic_2011} or fluid dynamics~\cite{nitikitpaiboon_ale_1993}, and they are also involved in multigrid solvers \cite{arnold_preconditioning_1997,hiptmair_multilevel_1999}.
Beyond numerical applications, potentials are fundamental to several variational formulations, for instance in kinetic or fluid electromagnetic models~\cite{low_lagrangian_1958, newcomb_lagrangian_1962} and they play a crucial role in quantum~\cite{feynman_quantum_1963} or relativistic physics \cite{rindler_introduction_1991}: a recent illustration is the elegant reformulation of the relativistic Vlasov-Maxwell equations as a wave-transport system involving Liénard-Wiechert potentials~\cite{bouchut_classical_2003,Bouchut2004}. Finally, their geometric properties expressed in Stokes' theorem are linked to key topological aspects of physical laws, particularly in electromagnetism~\cite{bossavit_computational_1998,gross_electromagnetic_2004}.

A difficulty arises when considering domains with cavities or tunnels, where harmonic fields with vanishing traces may exist~\cite{picard_boundary_1982, arnold_feec_2018}. Harmonic fields are unique because they are both irrotational and solenoidal, yet they may not admit scalar or vector potentials depending on their boundary conditions. Specifically, domains with \textit{cavities} have harmonic fields normal to the boundary that admit scalar potentials but no vector potentials, while domains with \textit{tunnels} have harmonic fields tangent to the boundary that admit vector potentials but no scalar potentials~\cite{cantarella_vector_2002}.

These cavities and tunnels correspond to the 2-chain and 1-chain homology groups of the domain, whose ranks (the Betti numbers) coincide with the dimensions of the spaces of normal and tangent harmonic vector fields, respectively~\cite{lee_ISM_2012,arnold_feec_2018}: in particular these dimensions are topological invariants, which means that they remain unchanged under homeomorphisms. 
For normal harmonic fields, a canonical construction has been proposed where scalar potentials are characterized by Laplace problems with boundary conditions fitted to the domain's cavities~\cite{picard_boundary_1982}. For tangent harmonic fields, a similar construction is more complex:
in domains diffeomorphic to a solid torus, a common approach is to use the toroidal angle as a (piecewise smooth) scalar potential~\cite{merkel_integral_1986,henneberg_combined_2021}, and this approach has been extended to domains that can be made simply connected by removing a finite number of disjoint \textit{cutting surfaces} \cite{foias1978remarques,picard_boundary_1982,ABDG_1998}. Such an assumption, however, imposes
topological constraints which exclude simple domains such as hollow tori 
\cite{benedetti_topology_2012}.
Moreover, a construction of proper vector potentials seemed to be lacking for tangent harmonic fields.

\bigskip
In this article we propose such a construction, summarized in Fig.~\ref{fig:visual_summary}: just as scalar potentials are characterized by homogeneous Laplace problems with boundary conditions fitted to the cavity surfaces, our vector potentials are characterized by homogeneous curl-curl problems with tangent boundary conditions fitted to closed curves that loop around the tunnels of the domain. These \textit{tunnel curves} form a basis family for the 1-chain homology group, and the associated vector potentials yield a basis for the space of tangent harmonic vector fields. To establish the linear independence of the resulting harmonic fields, our analysis relies on surfaces which are \textit{reciprocal} to the tunnel curves. These surfaces are reminiscent of the cutting surfaces involved in the construction of piecewise smooth scalar potentials, however they are not involved in the construction step and for this reason they do not need to be disjoint or render the complementary domain simply connected.
Instead, their main properties are that their own boundaries lie on the boundaries of the domain, and that they realize an intersection duality 
pairing with the tunnel curves. The fact that such a construction 
is possible for general domains is a consequence of the classical Poincaré-Lefschetz duality \cite{BGF_notes_1957,friedl_topology_2024}.

\begin{figure}[!h]
    \begin{center}

	\begin{tabular}{m{6cm}m{3cm}m{0.5cm}m{2.8cm}m{0.7cm}}

        \vspace{-15pt}
		\begin{equation*}
		\begin{cases}
        \Delta \phi = 0
			\\
		\phi = \text{cst on cavity surfaces}
		\end{cases}
		\end{equation*}
    
        &
    
    \hspace{2pt}
    \begin{tikzpicture}[scale=1.1]
            
    	\begin{scope}
	    	\shade[ball color=orange!95!white, opacity=0.9] (0,0) circle(1);
            \node at (1.25,-0.3) {$S$};
		
		    \clip (0,0) circle(1);

			\foreach \angle in {0,30,60,90,120,150} {
			\draw[thin, black, opacity=0.6, domain=-90:90, samples=50]
				plot ({cos(\angle)*cos(\x)}, {sin(\x)});
			}

			\draw[thin, black] (0,0) ellipse (1 and 0.5);
		\end{scope}

	\end{tikzpicture}
	&  
	$ \mspace{-20mu} \xrightarrow{\qquad} $ 
	&    
	\begin{tikzpicture}[scale=1]
		\shade[ball color=gray!20!white, opacity=0.9] (0,0) circle(1);

		\begin{scope}
			\clip (0,0) circle(1);

			\foreach \angle in {0,30,60,90,120,150} {
			\draw[thin, black, opacity=0.6, domain=-90:90, samples=50]
				plot ({cos(\angle)*cos(\x)}, {sin(\x)});
			}

			\draw[thin, black] (0,0) ellipse (1 and 0.5);
		\end{scope}

	\foreach \theta in {0,30,...,330} {
		\foreach \phi in {-60,-30,0,30,60} {
		\pgfmathsetmacro{\x}{cos(\theta)*cos(\phi)}
		\pgfmathsetmacro{\y}{sin(\phi)}
		\pgfmathsetmacro{\len}{0.25} 
		\draw[->, green!50!black, thick]
			(\x,\y) -- ({\x*(1+\len)}, {\y*(1+\len)});
		}
	}
	\end{tikzpicture}	
	&
    $\bgrad \phi$

    \\[1.5cm]

    \vspace{-10pt}
    \begin{equation*}
    \begin{cases}
        \bcurl \bcurl \bA = 0 
        \\
        \bn \times \bA = \text{cst along tunnel curves}
    \end{cases}
    \end{equation*}

	&


	\begin{tikzpicture}[scale=1.1]
    
    	\draw[fill=gray!10!white, opacity=0.9] (0,-0.75) ellipse (1cm and 0.3cm);
		\shade[gray!20, opacity=0.9] (-1,-0.75) -- (-1,0.75) arc (180:360:1cm and 0.3cm) -- (1,-0.75) arc (0:180:1cm and 0.3cm);
		\draw (-1,-0.75) -- (-1,0.75) (1,-0.75) -- (1,0.75);
		\shade[fill=white, opacity=0.9] (0,0.75) ellipse (1cm and 0.3cm);

		\foreach \h in {-0.5, 0, 0.5} {
			\draw[thin, black, opacity=0.6] (0,\h) ellipse (1cm and 0.3cm);
		}

        \node at (1.3,-0.3) {$\Gamma$};
		\foreach \h in {0} {
			\draw[thick, black, opacity=0.6] (0,\h) ellipse (1cm and 0.3cm);
		}


	\begin{scope}
        \foreach \h in {0} { 
			\foreach \angle in {-180, -135, -90, -45, 0} { 
				\pgfmathsetmacro{\xcoord}{cos(\angle)}
				\pgfmathsetmacro{\yoffset}{0.3 * sin(\angle)} 
				\pgfmathsetmacro{\startx}{\xcoord}
				\pgfmathsetmacro{\starty}{\h + \yoffset - 0.2} 

				\pgfmathsetmacro{\endx}{\xcoord}
				\pgfmathsetmacro{\endy}{\h + \yoffset + 0.25} 

				\draw[->, red, thick] (\startx, \starty) -- (\endx, \endy);
			}
		}
	\end{scope}
	\end{tikzpicture}

	&
    $ \mspace{-20mu}  \xrightarrow{\qquad} $
    &

    \hspace{2pt}
	\begin{tikzpicture}[scale=1]

    	\draw[fill=gray!10!white, opacity=0.9] (0,-0.75) ellipse (1cm and 0.3cm);
		\shade[gray!20, opacity=0.9] (-1,-0.75) -- (-1,0.75) arc (180:360:1cm and 0.3cm) -- (1,-0.75) arc (0:180:1cm and 0.3cm);
		\draw (-1,-0.75) -- (-1,0.75) (1,-0.75) -- (1,0.75);
		\shade[fill=white, opacity=0.9] (0,0.75) ellipse (1cm and 0.3cm);

		\foreach \h in {-0.5, 0, 0.5} {
			\draw[thin, black, opacity=0.6] (0,\h) ellipse (1cm and 0.3cm);
		}

		\begin{scope}
			\foreach \h in {-0.5, 0, 0.5} { 
			\foreach \angle in {-180, -135, -90, -45} { 
				\pgfmathsetmacro{\startx}{cos(\angle)}
				\pgfmathsetmacro{\starty}{\h + 0.3 * sin(\angle)}
				\pgfmathsetmacro{\endx}{cos(\angle + 30)} 
				\pgfmathsetmacro{\endy}{\h + 0.3 * sin(\angle + 30)} 

				\draw[->, blue, thick] (\startx, \starty) -- (\endx, \endy);
				}
			}
		\end{scope}
	\end{tikzpicture}
        &
        $\bcurl \bA$

    \end{tabular}
    \end{center}
    \caption{%
        Visual summary of the construction:
        on the top row, scalar potentials for normal harmonic fields (in green) are classically obtained
        by solving homogeneous Laplace equations with inhomogeneous boundary conditions (in orange) fitted to the cavity surfaces.
        On the bottom row, vector potentials for tangent harmonic fields (in blue) are obtained
        by solving homogeneous curl curl equations with inhomogeneous boundary conditions (in red) fitted to the tunnel curves.
    }
    \label{fig:visual_summary}
\end{figure}

\bigskip
An important feature of our construction is that it naturally applies to structure-preserving finite element spaces, where it allows to compute strong potentials for the harmonic fields. 
At the discrete level, we note that several authors have developed tree-based approaches as an alternative to cutting surfaces, which have been used to compute bases of divergence-free FEM spaces in \cite{hecht_construction_1981,AR_basis_2024}, vector potentials of (non harmonic) prescribed vector fields in \cite{rapetti_discrete_2003}, 
and discrete \textit{loop fields} in \cite{AR_construction_2013}, which may be used to generate a basis for tangent harmonic fields.

Under some assumptions which are fulfilled by Whitney~\cite{whitney_git_1957,bossavit_whitney_1988}, higher order Nédélec elements~\cite{nedelec_mixed_1980,nedelec_new_1986} and more generally by the discrete de Rham sequences constructed by Hiptmair~\cite{hiptmair_canonical_1999} and Arnold, Falk and Winther~\cite{AFW_2006,AFW_2010} in the Finite Element Exterior Calculus (FEEC) framework, our approach defines discrete potentials characterized by their geometric degrees of freedom \cite{bossavit_computational_1998,bochev_principles_2006} fitted to discrete cavity surfaces or tunnel curves on the boundary. As for the vector potentials, our surface boundary conditions are similar to the \textit{collar fields} considered in \cite{Hiptmair_Ostrowski_2021}, and in general domains where tunnel curves are not readily provided by the domain parametrization, one may use the algorithms for constructing discrete \text{1-cycles} proposed in \cite{Hiptmair_Ostrowski_2002,Hiptmair_Ostrowski_2021,pitassi_generators_2025}.

Finally, our approach also yields a direct proof that the discrete harmonic spaces have the correct dimensions. Since our analysis does not rely on uniform stability properties for the commuting projections, it applies to a large class of structure-preserving finite elements equipped with geometric degrees of freedom, such as spline spaces  \cite{buffa_isogeometric_2011} in single or multi-patch settings, for which the construction of stable commuting projections is not always guaranteed.

\bigskip 
The article is organized as follows. In Section~\ref{sec:homharm} we begin by reminding the definition of the homology groups, together with that of the harmonic spaces in the $L^2$ de Rham complex, and we discuss some important properties
regarding their dimensions and potentials.
In Section~\ref{sec:construct} we then discuss our main assumption and describe our construction, which is based on a decomposition of the vector potentials in two parts: a smooth part that lifts a specific boundary condition fitted to the tunnel curves, and a correction part that guarantees that the resulting potential solves an homogeneous curl-curl problem. Section~\ref{sec:discpot} is finally devoted to studying a discrete version of our construction. After listing some properties that several structure-preserving finite element frameworks possess, and reviewing the case of the discrete normal harmonic fields, we show that a natural application of our approach yields a basis of the discrete tangent harmonic fields, in terms of discrete vector potentials associated to the tunnels of the domain. We finally verify that all the discrete harmonic / cohomology spaces have the correct dimensions, namely the Betti numbers of the domain.

\bigskip 
In our work, an important source of inspiration has been the geometric approach promoted since the 1980s by Alain Bossavit for the development of structure-preserving methods in computational electromagnetism \cite{bossavit_whitney_1988,bossavit_computational_1998,rapetti_discrete_2003}. 
For instance, a construction of discrete vector potentials for magnetic fields satisfying specific flux conditions is described in \cite{bossavit_computational_1998}, within the framework of Whitney elements. We learned of Alain Bossavit's passing as we were writing this article, which we dedicate to his memory.

\section{Homology groups and harmonic function spaces} 
\label{sec:homharm}

Throughout the article, $\Omega$ is an open bounded domain in $\RR^3$ with Lipschitz boundary, in the usual sense that $\partial \Omega$ is locally the graph of a Lipschitz mapping with $\Omega$ on one side only.
For simplicity, we assume that $\Omega$ is connected.

\subsection{Harmonic fields in the Hilbert de Rham complex}

Let us start by reminding how the Hilbert spaces of harmonic functions 
are defined in $\Omega$: one first considers the Hilbert de Rham complex~\cite{bossavit_computational_1998,AFW_2010}
\begin{equation} \label{dR}
  0 
  \xrightarrow{~~} H^1(\Omega) 
  \xrightarrow{~ \bgrad ~} H(\bcurl;\Omega) 
  \xrightarrow{~ \bcurl ~} H(\Div; \Omega) 
  \xrightarrow{~ \Div ~} L^2(\Omega)
  \xrightarrow{~~} 0 
\end{equation}
for which the dual complex, made of the $L^2$ adjoint operators 
and their respective domains, is
\begin{equation} \label{dR_0}
  0 
  \xleftarrow{~~} L^2(\Omega)
  \xleftarrow{~ -\Div ~} H_0(\Div; \Omega) 
  \xleftarrow{~ \bcurl ~} H_0(\bcurl;\Omega) 
  \xleftarrow{~ -\bgrad ~} H^1_0(\Omega) 
  \xleftarrow{~~} 0.
\end{equation}
Following~\cite{AFW_2006}, we rewrite the primal complex 
using a generic notation
\begin{equation} \label{dR_V}
  0 
  \xrightarrow{~~} V^0
  \xrightarrow{~ d^0 ~} V^1
  \xrightarrow{~ d^1 ~} V^2
  \xrightarrow{~ d^2 ~} V^3
  \xrightarrow{~~} 0 
\end{equation}
and we denote the kernel and range spaces by
$$
\frZ^k := \{v \in V^k : d^k v = 0\}, 
\qquad \frB^k := d^{k-1} V^{k-1}.
$$
The dual complex is denoted as 
\begin{equation} \label{dR_0V}
  0 
  \xleftarrow{~~} V^*_0
  \xleftarrow{~ d^*_1 ~} V^*_1
  \xleftarrow{~ d^*_2 ~} V^*_2
  \xleftarrow{~ d^*_3 ~} V^*_3
  \xleftarrow{~~} 0
\end{equation}
with similar notation for the adjoint kernel and range spaces,
i.e., $\frZ^*_k := \{v \in V^*_k : d^*_k v = 0\}$ and 
$\frB^*_k := d^*_{k+1} V^*_{k+1}$.
The harmonic spaces are then formally defined as 
\begin{equation} \label{frH}
    \frH^k := \frZ^k \cap (\frB^k)^\perp 
        = \frZ^k \cap \frZ^*_k
\end{equation}
which yields for the Hilbert de Rham complex above:
\begin{equation}
  \label{frH_expl}
  \left\{ \begin{aligned}
    &\frH^0 = \RR, 
    \\
    &\frH^1 = H(\bcurl 0; \Omega) \cap H_0(\Div 0; \Omega),
    \\
    &\frH^2 = H(\Div 0; \Omega) \cap H_0(\bcurl 0; \Omega),
    \\
    &\frH^3 = \{0\}.
  \end{aligned} \right.
\end{equation}
Here, we have used standard concise notation for the kernel spaces
of the curl and divergence operators with or without boundary conditions, see e.g.~\cite{ACL_2018}. 
In \eqref{frH_expl} we see two spaces of harmonic vector fields: $\frH^1$
consists of harmonic fields which are \textit{tangent} to the boundary (also called Dirichlet fields), whereas $\frH^2$ consists of fields which are \textit{normal} to it (also called Neumann fields).
Harmonic fields may also be defined with mixed boundary conditions: we refer to \cite{pauly_hilbert_2022} for a study in the framework of general Hilbert complexes.
If the domain is contractible (i.e., simply connected with a connected boundary), the above spaces of harmonic vector fields must vanish as a consequence of the Poincaré lemma~\cite{bossavit_computational_1998}. In a general domain however this is not the case.

\subsection{Singular chains and homology groups}

The works of de Rham~\cite{de_rham_varietes_1955}
have established profound connections between harmonic spaces 
and homology groups. Here we follow~\cite{friedl_topology_2024} and consider singular homology groups in the compact domain $\overline{\Omega}$ (a topological manifold). We remind that a singular $k$-simplex (with $0 \le k \le 3$ here) is a continuous map $\sigma: \Delta_k \to \overline{\Omega}$ defined on the standard $k$-simplex $\Delta_k$, and that a (singular) $k$-chain is a formal linear combination of such singular simplices, with coefficients in $\ZZ$. Denoting by $\cC_k(\overline{\Omega})$ the corresponding group of $k$-chains, and by $\partial: \cC_k(\overline{\Omega}) \to \cC_{k-1}(\overline{\Omega})$ the (singular) boundary operator, we obtain a chain complex by considering the sequence
\begin{equation} \label{sing_chain}
  0
  \xleftarrow{~~} \cC_0(\overline{\Omega})
  \xleftarrow{~ \partial ~} \cC_1(\overline{\Omega})
  \xleftarrow{~ \partial ~} \cC_2(\overline{\Omega})
  \xleftarrow{~ \partial ~} \cC_3(\overline{\Omega})
  \xleftarrow{~~} 0.
\end{equation}
Here, the order from right to left corresponds to the orientation of the dual complex \eqref{dR_0}. Denoting by $\cZ_k(\overline{\Omega})$ and $\cB_k(\overline{\Omega})$
the $k$-cycles and $k$-boundaries, which are
the kernels and ranges (in $\cC_k(\overline{\Omega})$) of the boundary operator,
and observing that the inclusion $\cB_k(\overline{\Omega}) \subseteq \cZ_k(\overline{\Omega})$ follows from the sequence property $\partial \circ \partial = 0$, one next defines the $k$-th homology group as the quotient
\begin{equation}
    \label{hom}
    \cH_k(\overline{\Omega}) := \cZ_k(\overline{\Omega}) \quotient \cB_k(\overline{\Omega}).
\end{equation}
Thus, an element of $\cH_k(\overline{\Omega})$ is a class 
$[\sigma] = \{\sigma + \partial \tau : \tau \in \cC_{k+1}(\overline{\Omega})\}$ 
of closed $k$-chains $\sigma$ (closed in a manifold sense, not a topological one),
i.e. cycles, modulo boundary chains. Intuitively, this corresponds to the set of points (closed by convention), closed curves or closed surfaces that can be obtained from each other by continuous deformations within $\overline{\Omega}$. Note that in $\RR^3$ no bounded volume can be closed in a manifold sense. 
An illustration is provided in Figure~\ref{fig:domain}.

\begin{figure}[!h]
    \begin{center}
    \vspace{-3cm}
    \def\svgwidth{400pt}
    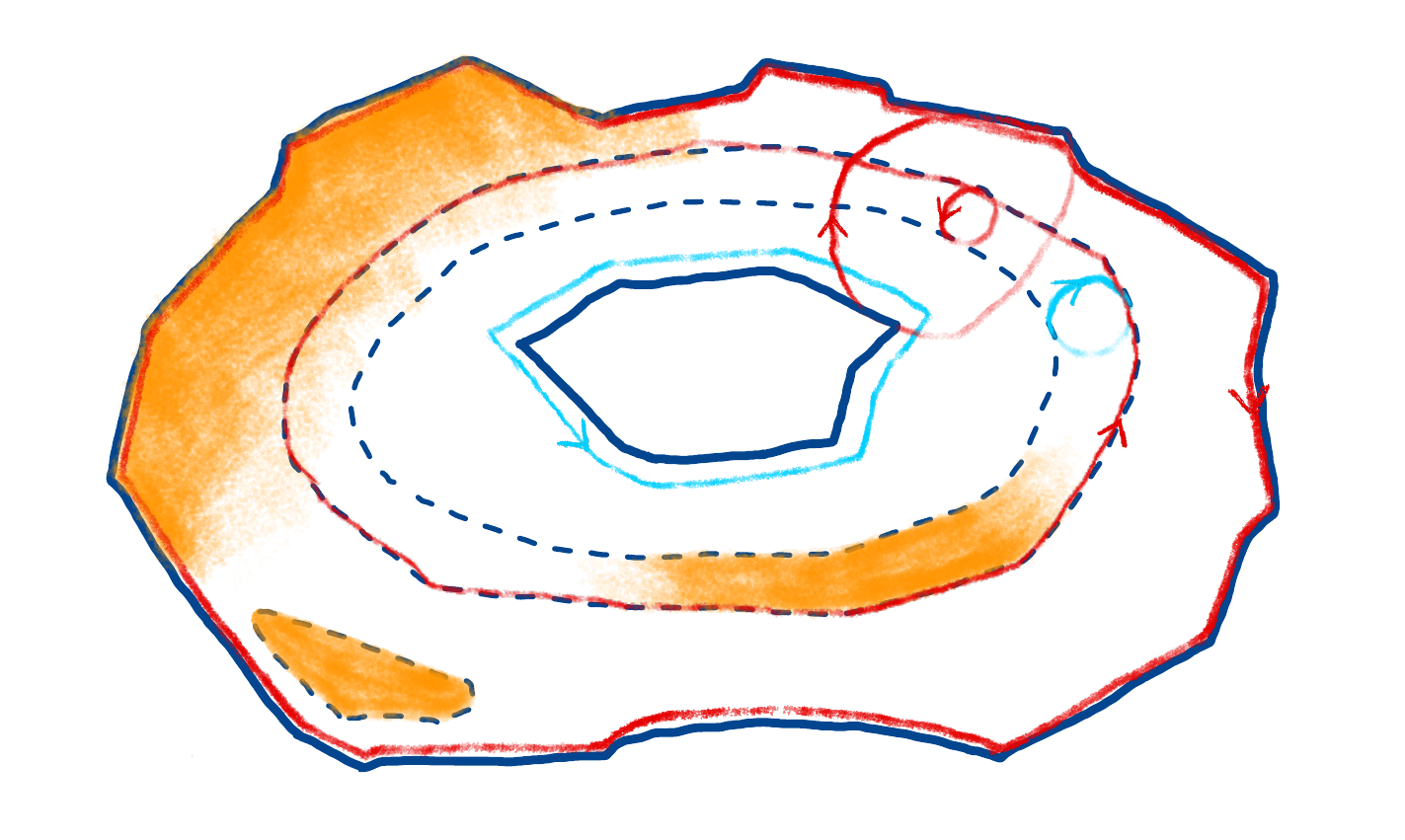
    \end{center}
\caption{%
    Example of a non-contractible domain $\Omega$
    with two cavities ($\beta_2 = 2$) and two tunnels ($\beta_1 = 2$): 
    one that goes through the domain, and another one 
    that is also a cavity and circles around the first tunnel. 
    The domain boundary has three connected components which 
    are closed surfaces (in orange, only partially drawn for the sake of visibility): 
    $S_0$ is the ``outer'' boundary which bounds the unbounded component of $\RR^3 \setminus \Omega$, 
    $S_1$ bounds the contractible cavity,
    and $S_2$ bounds the inner tunnel.  
    The last two represent independent classes of the 
    2-chain homology group $\cH_2(\overline{\Omega})$, while
    the first one is homologous to $-(S_1+S_2)$.
    Finally, the closed curves $\Gamma_i \subset \partial \Omega$, $i = 1, 2$, (in light blue) represent two independent classes of the 1-chain homology group $\cH_1(\overline{\Omega})$. They admit reciprocal surfaces $\Sigma_i$ as described 
    in Assumption~\ref{as:Gamma}, visualized by their boundary $\partial \Sigma_i = \cup_a \tilde{\Gamma}_{i,a}$ with connected curves 
    $\tilde{\Gamma}_{i,a} \subset \partial \Omega$, drawn in red.%
}
\label{fig:domain}
\end{figure}

One can verify that $\cH_k(\overline{\Omega})$ is a free abelian group, akin to a vector space with integer coefficients.
Its rank is called the $k$-th \textit{Betti number} of $\Omega$, and denoted
\begin{equation}
    \label{betti}
    \beta_k := \rank(\cH_k(\overline{\Omega})).
\end{equation}

\begin{remark} \label{rem:cHOmega}
    One may also consider singular chains in the open domain $\Omega$,
    and the corresponding homology groups $\cH_k(\Omega)$.
    According to \cite[Cor.~44.7]{friedl_topology_2024}, the manifolds $\overline{\Omega}$ and $\Omega$ are homotopy equivalent, hence their homology groups are isomorphic.
\end{remark}
    
The extremal numbers are simple:
we always have $\beta_3 = 0$ since there are no closed volumes (three-dimensional cycles) in $\Omega$, and $\beta_0$ is the number of connected components, hence $\beta_0 = 1$ for a connected domain.
The intermediate Betti numbers are more interesting and characterize the lack of contractibility of $\Omega$.
In particular, $\beta_2$ is the number of cavities in the domain, i.e., the number of connected components of $\RR^3 \setminus \Omega$ which are bounded, and $\beta_1$ is the number of tunnels. 
Notice that tunnels may go through $\Omega$ as in handles, or be entirely surrounded by it, which makes them cavities as well: an example of both types is shown in Figure~\ref{fig:domain}.

In the sequel we will denote by $S_i$, $i = 0, \dots \beta_2$,
the connected components of the boundary $\partial \Omega$
(oriented by the normal vector pointing outside of $\Omega$).
They are all closed surfaces: we let $S_0$ be the 
boundary of the single unbounded connected component of 
$\RR^3 \setminus \Omega$, so that for every $i = 1, \dots \beta_2$,
$S_i$ bounds a single cavity surrounded by the domain $\Omega$:
these surfaces form a basis of the 2-chain homology group $\cH_2$. Observe that $S_0$ is homologous to $-(\sum_{i=1}^{\beta_2} S_i)$ since the difference bounds the domain itself.

Similarly, we will denote by $\Gamma_i$, $i = 1, \dots \beta_1$,
a family of oriented closed curves that generate the 1-chain homology group $\cH_1$. These curves will be chosen on the boundary $\partial \Omega$, such that each one circles once around a different tunnel: a precise description will be discussed in Section~\ref{sec:assum} below.

\subsection{The de Rham theorem for the $L^2$ de Rham complex}

An important result~\cite{AFW_2006,arnold_feec_2018}, which essentially follows from de Rham's theorem~\cite{de_rham_analysis_situs_1931,cartan_algebraic_1948,lee_ISM_2012}, is that the dimensions of the harmonic spaces \eqref{frH} coincide with the Betti numbers of the domain.
\begin{theorem} \label{th:dRL2}
For $k = 0, \dots 3$, the following relation holds:
\begin{equation}
  \label{h_b}
  \dim(\frH^k) = \beta_k.
\end{equation}
\end{theorem}

Since the classical version of de Rham's theorem is formulated for the smooth complex, we give here a detailed argument which uses a result of Costabel and McIntosh~\cite{Costabel_McIntosh_2010}, suggested to us by Douglas Arnold.

\begin{proof}
    According to Remark~\ref{rem:cHOmega} we may consider the homology groups on $\Omega$ which, as an open subset of $\RR^3$, is a smooth manifold (without boundary). We can then apply the de Rham theorem as formulated in~\cite[Th.~18.14]{lee_ISM_2012}, which states that the $k$-chain cohomology groups 
    $\cH^k(\Omega;\RR)$ (which are real vector spaces isomorphic to the dual spaces $(\cH_k(\Omega))'$ of linear forms on $\cH_k(\Omega)$, and hence of the same dimensions $\beta_k$) are isomorphic to the cohomology groups (spaces) of the smooth de Rham complex
    \begin{equation} \label{sm_dR}
    0 
    \xrightarrow{~~} C^\infty(\Omega) 
    \xrightarrow{~ \bgrad ~} C^\infty(\Omega)^3
    \xrightarrow{~ \bcurl ~} C^\infty(\Omega)^3
    \xrightarrow{~ \Div ~} C^\infty(\Omega)
    \xrightarrow{~~} 0.
    \end{equation}
    Rewriting the latter as $ \dots \xrightarrow{} V^k_{\infty} \xrightarrow{~ d^k ~} V^{k+1}_{\infty} \xrightarrow{} \dots $
    and denoting by
    $\frZ^k_{\infty} := \{v \in V^k_{\infty} : d^k v = 0\}$
    and 
    $\frB^k_{\infty} := d^{k-1} V^{k-1}_{\infty}$
    the associated kernel and range spaces, one has then
    \begin{equation}
        \dim\big(\frZ^k_{\infty} \quotient \frB^k_{\infty}\big) = \beta_k.
    \end{equation}
    Moreover, a Poincaré duality theorem holds
    (\cite[Pbm.~18.7]{lee_ISM_2012}), which states that 
    the cohomology groups of the compactly supported 
    de Rham sequence (indexed with reverse order) have the same dimensions. 
    Namely, writing 
    \begin{equation} \label{cs_dR}
    0 
    \xleftarrow{~~} C^\infty_c(\Omega) 
    \xleftarrow{~ \Div ~} C^\infty_c(\Omega)^3
    \xleftarrow{~ \bcurl ~} C^\infty_c(\Omega)^3
    \xleftarrow{~ \bgrad ~} C^\infty_c(\Omega)
    \xleftarrow{~~} 0
    \end{equation}
    as 
    $ \dots \xleftarrow{} V^\infty_{c, k-1} \xleftarrow{~ d^*_{k} ~} V^\infty_{c, k} \xleftarrow{} \dots $
    and denoting the associated kernel and range spaces by
    $\frZ^\infty_{c,k} := \{v \in V^\infty_{c,k} : d^*_k v = 0\}$
    and 
    $\frB^\infty_{c,k} := d^*_{k+1} V^{\infty}_{c,k+1}$,
    one has 
    \begin{equation}
    \dim\big(\frZ^\infty_{c,k} \quotient \frB^\infty_{c,k}\big)
        = \dim\big(\frZ^k_{\infty} \quotient \frB^k_{\infty}\big).
    \end{equation}
    To link the above relations with the Hilbert harmonic spaces \eqref{frH}, 
    we next invoke two results from~\cite{Costabel_McIntosh_2010}.
    The first one is Theorem~1.1, which states that the 
    cohomology spaces of the sequences
    \begin{equation} \label{Vc_and_Hs}
    H^{-1}_{\overline{\Omega}}(\RR^3)
        \xleftarrow{~ d^*_{k} ~} H^0_{\overline{\Omega}}(\RR^3)
        \xleftarrow{~ d^*_{k+1} ~} H^1_{\overline{\Omega}}(\RR^3)
    \qquad \text{ and } \qquad 
    V^{\infty}_{c, k-1} 
        \xleftarrow{~ d^*_{k} ~} V^\infty_{c, k} 
        \xleftarrow{~ d^*_{k+1} ~} V^{\infty}_{c,k+1}
    \end{equation}
    can be represented with the same functions.
    Here, $H^s_{\overline{\Omega}}(\RR^3)$ is the space 
    of distributions in $H^s(\RR^3)$ with support in $\overline{\Omega}$,
    and the differential operators are meant in the sense of distributions
    over $\RR^3$. 
    We thus have 
    \begin{equation}
        \label{dim_CMI}
        \dim\big(\frZ^*_{\overline{\Omega},k} \quotient \frB^*_{\overline{\Omega},k}\big) = 
        \dim\big(\frZ^\infty_{c,k} \quotient \frB^\infty_{c,k}\big)
    \end{equation}
    with
    $$
    \frZ^*_{\overline{\Omega},k} 
        := \{v \in L^2_{\overline{\Omega}}(\RR^3) : d^*_k v = 0\}
    \qquad \text{ and } \qquad 
    \frB^*_{\overline{\Omega},k} := d^*_{k+1} H^1_{\overline{\Omega}}(\RR^3).
    $$
    The second useful result from~\cite{Costabel_McIntosh_2010} is Corollary~4.7, 
    which allows to write
    \begin{equation*}
    d^*_{k+1}H^1_{\overline{\Omega}}(\RR^n) 
    = d^*_{k+1} H_{\overline{\Omega}}(d^*_{k+1};\RR^n)        
    \end{equation*}
    where the last space is defined as
    $H_{\overline{\Omega}}(d^*_{k+1};\RR^n) 
    := \{v \in L^2_{\overline{\Omega}}(\RR^3) : 
    d^*_{k+1} v \in L^2_{\overline{\Omega}}(\RR^3) \}$.
    Since this space consists of the $L^2(\RR^3)$ functions which have a differential in $L^2$ and vanish outside of $\Omega$, 
    it coincides with the space $V^*_{k+1}$ from the complex 
    with homogeneous boundary conditions \eqref{dR_0}.
    As a result, we find 
    $$
    \frB^*_{\overline{\Omega},k} 
        = d^*_{k+1} H^1_{\overline{\Omega}}(\RR^3)
        = d^*_{k+1} V^*_{k+1} = \frB^*_k,
    $$
    and we also observe that
    $$
    \frZ^*_{\overline{\Omega},k} 
    = \{ v \in H_{\overline{\Omega}}(d^*_{k};\RR^n) : d^*_k v = 0\} 
    = \{ v \in V^*_k : d^*_k v = 0\} = \frZ^*_k.
    $$
    It follows that \eqref{dim_CMI} rewrites as 
    \begin{equation}
        \label{dim_H*dRc}
        \dim\big(\frZ^*_{k} \quotient \frB^*_{k}\big) = 
        \dim\big(\frZ^\infty_{c,k} \quotient \frB^\infty_{c,k}\big).
    \end{equation}
    The proof is finally completed by observing that the orthogonal projection 
    onto $(\frB^*_k)^\perp$ induces an isomorphism 
    between the quotient space $\frZ^*_k \quotient \frB^*_k$ 
    and the harmonic space $\frH^k = \frZ^*_k \cap (\frB^*_k)^\perp$.
\end{proof}

\subsection{Scalar and vector potentials}

An important property of the complexes \eqref{dR} and \eqref{dR_0} is that the range spaces $\frB^k$ and $\frB^*_k$ are closed in $L^2(\Omega)$~\cite{AFW_2006}. As a result, one can write a general Hodge-Helmholtz decomposition of the form 
\begin{equation}
    L^2(\Omega) = \frB^k \poplus \frH^k \poplus \frB^*_k
\end{equation}
which yields two $L^2$-orthogonal decompositions for vector fields: one that corresponds to $k=1$,
\begin{equation} \label{HH_dec_1}
    L^2(\Omega) = \bgrad H^1(\Omega) \poplus \frH^1 \poplus \bcurl H_0(\bcurl;\Omega)
\end{equation}
and another one corresponding to $k = 2$,
\begin{equation} \label{HH_dec_2}
    L^2(\Omega) = \bcurl H(\bcurl;\Omega) \poplus \frH^2 \poplus \bgrad H^1_0(\Omega).
\end{equation}
With regard to potentials, the question is then to know whether harmonic vector fields belong to some range space. It is clear from \eqref{HH_dec_1} that the tangential harmonic fields in $\frH^1$ have no scalar potential, and from \eqref{HH_dec_2} we see that the normal harmonic fields in $\frH^2$ have no vector potential. But the reverse may be true, and a simple argument shows that it is indeed.

\begin{proposition}
The following inclusions hold:
\begin{equation}
    \label{H0c0_sub_grad}
    H_0(\bcurl 0;\Omega) \subseteq \bgrad H^1(\Omega),
\end{equation}
and 
\begin{equation}
    \label{H0d0_sub_curl}
    H_0(\Div 0;\Omega) \subseteq \bcurl H(\bcurl;\Omega).
\end{equation}
In particular, every field in $\frH^2$ admits a scalar potential
in $H^1(\Omega)$, and every field in $\frH^1$ admits a vector potential
in $H(\bcurl;\Omega)$.
\end{proposition}

\begin{proof}
If $\bv \in H_0(\bcurl 0;\Omega)$, resp. $\bv \in H_0(\Div 0;\Omega)$,
then letting $\bar \bv$ be its extension by 0 
on an open ball $B_R$ that contains $\overline \Omega$ yields 
$\bar \bv \in H_0(\bcurl 0;B_R)$, resp. $\bar \bv \in H_0(\Div 0;B_R)$. 
Since $B_R$ is contractible this allows to find $\bar \phi \in H^1_0(B_R)$
such that $\bar \bv = \bgrad \bar \phi$, resp. $\bar \bA \in H_0(\bcurl;B_R)$
such that $\bar \bv = \bcurl \bar \bA$. After restriction to $\Omega$, this yields $\bv = \bgrad \phi$ with $\phi := \bar \phi|_{\Omega} \in H^1(\Omega)$,
respectively $\bv = \bcurl \bA$ with $\bA := \bar \bA|_{\Omega} \in H(\bcurl;\Omega)$.
The last statement is a direct consequence of the definitions \eqref{frH}--\eqref{frH_expl}.
\end{proof}

\begin{remark}
    For a characterization of smooth fields with scalar or vector potentials, 
    see~\cite{cantarella_vector_2002}.
\end{remark}

\subsection{Scalar potentials for the normal harmonic fields}
\label{sec:scapot}

For the space $\frH^2$ of normal harmonic fields, a natural approach is to consider scalar potentials with constant values on each connected part of the boundary.
Following~\cite{picard_boundary_1982} and~\cite{ABDG_1998}, define
\begin{equation} \label{H1S}
    H^1_S(\Omega) := \{ \vp \in H^1(\Omega) : \vp|_{S_0} = 0 ~ \text{ and } ~ \vp|_{S_i} = \text{constant}, ~ 1 \le i \le \beta_2 \}
\end{equation}
and for each $1 \le i \le \beta_2$, 
let $\phi_i \in H^1_S(\Omega)$ solve the Poisson problem 
\begin{equation}
    \label{phi_i_pbm}
    \sprod{\bgrad \vp}{\bgrad \phi_{i}} = \vp|_{S_i}
        \qquad \forall \vp \in H^1_S(\Omega).
\end{equation}
This problem is clearly well-posed thanks to Poincaré's inequality 
on $H^1_S(\Omega)$, which holds since $S_0 \neq \emptyset$, 
and we have the following result (see, e.g.~\cite[Prop.~3.18]{ABDG_1998}).
\begin{proposition} \label{prop:gradbasis}
    The vector fields $\bv_i := \bgrad \phi_i$, $1 \le i \le \beta_2$, form
    a basis of $\frH^2$, and it holds
    \begin{equation}
        \label{fluxes_gradphi}
        \int_{S_j} \bn \cdot \bv_i = \delta_{i,j}, 
        \quad 1 \le j \le \beta_2.
    \end{equation}
\end{proposition}

\begin{proof}
    We remind the argument for the sake of completeness.
    The first observation is that 
    \begin{equation} \label{gradH1c}
        \bgrad H^1_S(\Omega) \subset H_0(\bcurl;\Omega)
    \end{equation}
    which is easy to verify: for all $1 \le j \le \beta_2$, 
    fix a function $\psi_j \in H^1_S(\Omega)$
    that is constant on a (disconnected) neighborhood of $\partial \Omega$
    and satisfies
    \begin{equation} \label{psi_i}
        \psi_j|_{S_0} = 0 \quad \text{ and } \quad
        \psi_j|_{S_l} = \delta_{j,l}, \quad 1 \le l \le \beta_2.
    \end{equation}
    For any $\vp \in H^1_S(\Omega)$, the function 
    $
    \bgrad \big( \vp - \sum_{j=1}^{\beta_2} (\vp|_{S_j})\psi_j\big)
    $
    then belongs to $H_0(\bcurl;\Omega)$ as the gradient of a function 
    in $H^1_0(\Omega)$, and its trace clearly coincides with that of $\bgrad \vp$.
    This essentially shows that $\bv_i \in \frH^2$: indeed
    one obviously has $\bcurl \bv_i = 0$ and 
    $\Div \bv_i = 0$ follows by taking $\vp \in C^\infty_c(\Omega)$ 
    in \eqref{phi_i_pbm}.
    To verify \eqref{fluxes_gradphi} we then use again the functions $\psi_j$
    and compute     
    \begin{equation*}
    \int_{S_j} \bn \cdot \bv_i
    = \int_{\partial \Omega} \psi_j ~ \bn \cdot \bv_i
    = \sprod{\bgrad \psi_j}{\bv_i}
    = \sprod{\bgrad \psi_j}{\bgrad \phi_i} = \psi_j|_{S_i} = \delta_{i,j}. 
    \end{equation*}
    This in particular shows that the vector fields $\bv_i$ are linearly
    independent. The spanning property then follows from Theorem~\ref{th:dRL2}.
\end{proof}

The next section describes a construction of vector potentials for tangent harmonic fields that extends the previous ideas.

\section{Construction of vector potentials for tangent harmonic fields}
\label{sec:construct}

\subsection{Tunnel curves and reciprocal surfaces}
\label{sec:assum}

Motivated by the discussion in~\cite[Sec.~3]{BGF_notes_1957}, our construction relies on the assumption that the domain $\Omega$ admits piecewise smooth ``tunnel curves'' $\Gamma_i$, which are closed curves 
lying on the boundary $\partial \Omega$, and ``reciprocal surfaces'' $\Sigma_i \subset \overline{\Omega}$, $1 \le i \le \beta_1$, whose boundaries $\partial \Sigma_i$ also lie on the domain boundary and intersect the tunnel curves in a dual manner, 
in the sense that the number of intersection points of
$\Gamma_i$ and $\Sigma_j$ satisfies 
\begin{equation}
    \label{I_duality}
    I(\Gamma_i, \Sigma_j) = \delta_{i,j}, \quad 1 \le i,j \le \beta_1.
\end{equation}

Aside from regularity considerations, the fact that such a property holds for general domains follows from classical results in algebraic topology, and in particular from the geometric realization of Poincaré-Lefschetz duality through intersection pairings.
This important duality result states that the $k$-chain homology groups 
$\cH_k(\overline{\Omega})$ are isomorphic to the $(3-k)$-chain relative homology groups $\cH_{3-k}(\overline{\Omega}, \partial \Omega)$, which are the homology groups of the singular chains \textit{modulo the boundary} -- 
two chains being equal modulo the boundary if their difference lies on $\partial \Omega$, see e.g. \cite{lefschetz_introduction_1949, BGF_notes_1957}.
A modern exposition of these results can be found in the detailed
lecture notes of Stefan Friedl \cite{friedl_topology_2024}: applying 
Theorem 194.14 and Proposition 198.2 
one finds indeed that there exists 
a non-singular pairing between the above 
(torsion-free) homology groups,
\begin{equation}
    \label{asQ}
Q = Q^{as}_{\overline{\Omega},\emptyset,\partial \Omega}: 
\cH_k(\overline{\Omega}) \times \cH_{3-k}(\overline{\Omega}, \partial \Omega) \to \ZZ ~,
\end{equation}
called \textit{asymmetric intersection pairing} 
for the reason that for any pair of cycles $\gamma \in \cZ_k(\overline{\Omega})$ and $\sigma \in \cZ_{3-k}(\overline{\Omega}, \partial \Omega)$ which intersect nicely, $Q([\gamma], [\sigma])$ coincides with the algebraic intersection number of $\gamma$ and $\sigma$. 

With $k=1$, we thus find that for any 
collection of (reasonably smooth) closed curves $\Gamma_i$
representing a basis of $\cH_1(\overline{\Omega})$,
there exists a collection of surfaces 
$\Sigma_i$ which are closed modulo to the boundary,
represent a basis of 
$\cH_{2}(\overline{\Omega}, \partial \Omega)$
and are such that 
$Q([\Gamma_i], [\Sigma_j]) = \delta_{i,j}$ for all $i,j$.
We then observe that the curves $\Gamma_i$ 
may be chosen to lie on the domain boundary
without changing their homology class $[\Gamma_i]$,
and that the surfaces fulfill the above requirements
(as long as they can be chosen to intersect nicely the curves,
which is not a restrictive assumption in a Lipschitz domain $\Omega$):
being closed modulo the boundary means indeed that 
$\partial \Sigma_i \subset \partial \Omega$,
and a relation of the form \eqref{I_duality} 
follows directly from the $Q$-duality of the homology classes, 
with the specification that $I$ denotes 
the \textit{algebraic} intersection number of the singular chains,
which weighs each intersection with a sign corresponding to their respective
orientations.

\begin{figure}[ht!]
    \vspace{-5.5cm}
    \begin{center}
    \def\svgwidth{350pt}
    \input{inkscape_figs/mapping.pdf_tex}
    \end{center}
\caption{Illustration of a tunnel curve $\Gamma_i$ and its neighborhood $X_i$, parametrized with a piecewise smooth homeomorphism $T_i$ defined on the (straight) torus $\cO$, periodic along the variable~$\theta$.}
\label{fig:mapping_T}
\end{figure}

To formulate our precise working hypotheses, we now introduce a couple of definitions.

\begin{definition}
    \label{def:pwCk}
    A function $u$ defined on a bounded Lipschitz domain $M \subset \RR^d$ 
    is said piecewise $C^k$, $k \ge 0$, if there exists a decomposition 
    $\overline M = \cup_{a = 1}^{m} \overline M_{a}$ 
    into disjoint open cells $M_a$ with Lipschitz, piecewise $C^1$ boundaries, 
    such that for all $1 \le a \le m$, the restriction $u \vert_{M_a}$ 
    is in $C^k(\overline {M_{a}})$. If in addition $u$ is invertible 
    and $u^{-1}$ is piecewise $C^k$ over the cells $u(M_{a})$, $1 \le a \le m$,
    then we say that $u$ is a piecewise $C^k$ diffeomorphism.
\end{definition}
In the sequel we will refer to such a decomposition
$\cM = \{M_a : a = 1, \dots m\}$ as a \textit{mesh}, 
and we will denote by $C^k(\cM)$ the space of piecewise $C^k$ 
functions over $\cM$.

\begin{definition} \label{def:compsurf}
    We say that a surface $\Sigma$ and a mesh
    $\cM$ are compatible if there exists a decomposition 
    $\Sigma = \cup_{a = 1}^{n} \Sigma_{a}$ 
    made of smooth surfaces with corners \cite{lee_ISM_2012}, such that
    the interiors $\Sigma_{a} \setminus \partial \Sigma_{a}$
    are disjoint and included in 
    closed mesh cells $\overline M_{b}$ 
    for some $b = b(a) \in \{0, \dots m\}$.
\end{definition}

Our working hypotheses read then as follows.
\begin{assumption} 
    \label{as:Gamma} 
    There exist oriented curves
    $\Gamma_i \subset \partial\Omega$
    and surfaces $\Sigma_i \subset \overline{\Omega}$,
    $i=1,\ldots \beta_1$,
    with the following properties:
    \begin{itemize}
    
        \item        
        Each curve $\Gamma_i$ is closed and piecewise smooth, in particular
        there is a homeomorphism
        \begin{equation}
            T_i: \cO \ni(r,\theta,z) \mapsto \bx\in X_i
        \end{equation}
        between the parametric domain $\cO := [0, 1)\times\TT\times(-1, 1)$,
        where $\TT := \RR \quotient \ZZ$,
        and some subdomain $X_i \subset \overline{\Omega}$,
        such that 
        \begin{itemize}
            \item[(i)]
            the surface $\{r = 0\}$ parametrizes the boundary part of $X_i$, i.e., 
            \begin{equation}
                X_i \cap \partial\Omega = T_i(\{0\}\times\TT\times(-1, 1))
            \end{equation}
            
            \item[(ii)]
            the curve $\hat \Gamma := \{r = 0, z = 0\} \cong \TT$
            parametrizes $\Gamma_i \subset X_i$, i.e.,
            \begin{equation}
                \Gamma_i = T_i(\{0\}\times\TT\times\{0\})
            \end{equation}
            and $\Gamma_i$ is oriented by the
            unit tangent vector $\btau_{\Gamma_i}$ along 
            increasing values of $\theta$

            \item[(iii)]
            the mapping $T_i$ is a piecewise $C^\infty$ diffeomorphism
            in the sense of Definition~\ref{def:pwCk}.
            Letting $\cO_{i,a}$, $a = 1, \dots m_i$, be the associated 
            open mesh cells of $\cO$, we denote by
            \begin{equation}
                \label{Xia}
                \cX_i = \{X_{i,a} : a = 0, \dots m_i\}
                \quad \text{with} \quad 
                \begin{cases}
                    X_{i,0} := \Omega \setminus \overline{X}_i,
                    \\
                    X_{i,a} := T_i(\cO_{i,a}), ~~ a = 1, \dots m_i,
                \end{cases}
            \end{equation} 
            the corresponding mesh of $\Omega$.
        \end{itemize}
                
        \item 
        Each surface $\Sigma_i$ is Lipschitz,
        compatible with all the meshes $\cX_j$, $j = 1, \dots \beta_1$,
        in the sense of Definition~\ref{def:compsurf},
        and its (manifold) boundary $\partial \Sigma_i$ is included in $\partial \Omega$.
    \end{itemize}

    Finally, the above curves and surfaces 
    are reciprocal in the sense that $\Gamma_i$ intersects
    $\Sigma_j$ once if $i = j$, and zero times if $i \neq j$.
    More precisely, observing that 
    $\Gamma_i \cap \Sigma_j = \Gamma_i \cap \partial\Sigma_j 
    \subset X_i \cap \partial\Sigma_j$, 
    it holds
    \begin{equation}
        \label{XGprime}
        X_i \cap \partial\Sigma_j = 
        \begin{cases}
            T_i\big(\{0\}\times\{0\}\times(-1,1)\big) & \text{ if } i=j,
            \\
            \emptyset & \text{ if } i \neq j,
        \end{cases}
    \end{equation}
    and we orient each surface $\Sigma_i$ by letting
    its unit normal vector $\bn_{\Sigma_i}$ 
    at the intersection point $P_i = \Gamma_i \cap \Sigma_i = T_i(0,0,0)$
    satisfy $\bn_{\Sigma_i} \cdot \btau_{\Gamma_i} > 0$.
\end{assumption}

\begin{remark}
    In the above assumption, and in particular in \eqref{XGprime},
    we have considered for simplicity 
    that an intersection duality of the form \eqref{I_duality}
    can be realized with \textit{geometric} intersection numbers
    $I(\Gamma_i, \Sigma_j) = \#(\Gamma_i \cap \Sigma_j)$,
    which is the case for virtually any domain we can think of.
    In the general case however, this duality
    is only established with \textit{algebraic} intersection numbers
    as discussed above.
    Instead of \eqref{XGprime}, one should thus assume that for any $i,j$, 
    there exists distinct points 
    $\theta(i,j,p) \in \TT$, $1 \le p \le n(i,j)$
    (with $n(i,j)$ possibly zero), such that 
    \begin{equation}
        \label{XGprime_alg}
        X_i \cap \partial\Sigma_j = 
        \cup_{1 \le p \le n(i,j)} \, 
        T_i\big(\{0\}\times\{ \theta(i,j,p) \}\times(-1,1)\big)
    \end{equation}
    holds, with normal vectors $\bn_{\Sigma_i}$ and 
    indices at the intersection points 
    $P_{i,j,p} = T_i(0,\theta(i,j,p),0)$,
    \begin{equation}
        \label{indices_alg}
        {\rm ind}(i,j,p) := 
        {\rm sign}\Big( (\btau_{\Gamma_i} \cdot \bn_{\Sigma_j})(P_{i,j,p}) \Big) 
        \in \{-1, +1\},
    \end{equation}
    that satisfy the duality relation
    \begin{equation}
        \label{ind_duality}
        I(\Gamma_i,\Sigma_j) := 
        \sum_{p=1}^{n(i,j)} {\rm ind}(i,j,p) = \delta_{i,j}, 
        \quad 1 \le i,j \le \beta_1.
    \end{equation}
    In our analysis below we will consider this general case
    since it is handled by the same arguments.
\end{remark}

\begin{remark}
    The fact that the curves $\Gamma_i$ represent a basis of the first homology group $\cH_1(\overline{\Omega})$ follows from Assumption~\ref{as:Gamma} and properties of the intersection pairing \eqref{asQ}
    recalled at the beginning of Section~\ref{sec:assum}.
    Indeed, these curves being 1-cycles they represent 1-chain homology classes, and similarly the surfaces $\Sigma_i$ being closed
    modulo the boundary they represent relative homology classes 
    in $\cH_2(\overline{\Omega}, \partial \Omega)$.
    Their intersection duality assumed in 
    \eqref{XGprime} (or more generally in \eqref{XGprime_alg}--\eqref{ind_duality}) then yields $Q([\Gamma_i], [\Sigma_j]) = \delta_{i,j}$, which shows the linear independence of the classes $[\Gamma_i]$.
    Their generating property follows from the definition of $\beta_1$.
\end{remark}

\subsection{The surface flux functionals}

The main purpose of the reciprocal surfaces $\Sigma_i$ is to establish the linear independence of our tangent harmonic vector fields. 
To this end we first review a version of the usual Stokes formula for piecewise smooth curl-conforming fields. 
\begin{proposition}
    Let $\Sigma$ be an oriented surface compatible with a mesh $\cM$ 
    in the sense of Definition~\ref{def:compsurf}.
    Then the Stokes formula 
    \begin{equation}
        \label{Stokes_S}
        \int_{\Sigma} \bn \cdot \bcurl \bu = \int_{\partial \Sigma} \btau \cdot \bu
    \end{equation}
    holds for any $\bu \in C^\infty(\cM) \cap H(\bcurl;\Omega)$.
\end{proposition}

\begin{proof}
    Since any surface $\Sigma_{a}$, $a = 1, \dots n$, 
    involved in the decomposition of Definition~\ref{def:compsurf}
    is a smooth surface with corners and is included in some domain $\overline M_{b(a)}$ where $\bu$ is smooth, Stokes' theorem applies.
    Using the correspondence between vector fields and differential forms
    (see for instance \cite[Ch.~16]{lee_ISM_2012}),
    this yields 
    \begin{equation}
        \label{Stokes_Sia}
        \int_{\Sigma_{a}} \bn \cdot \bcurl \bu = \int_{\partial \Sigma_{a}} \btau \cdot \bu|_{M_{b(a)}}, \qquad 1 \le a \le n.
    \end{equation}
    Summing over $a$ then yields a collection of surface integrals which sum to the left-hand side integral in \eqref{Stokes_S}, and a collection of line integrals on the right-hand side: 
    some corresponding to interior curves $\gamma_{a,a'} = \partial \Sigma_{a} \cap \partial \Sigma_{a'}$ 
    with $a \neq a'$, and others corresponding to boundary curves
    $\gamma_{a} = \partial \Sigma_{a} \cap \partial \Sigma$.
    On interior curves which appear twice with opposite orientations, we have
    $$
    \int_{\gamma_{a,a'}} \btau_{\partial\Sigma_{a}} \cdot \bu|_{M_{b(a)}} 
    + \int_{\gamma_{a,a'}} \btau_{\partial\Sigma_{a'}} \cdot \bu|_{M_{b(a')}} = 
    \int_{\gamma_{a,a'}} \btau_{\partial\Sigma_{a}} \cdot (\bu|_{M_{b(a)}}-\bu|_{M_{b(a')}}) = 0
    $$
    where the last equality follows from the continuity of the tangential traces of piecewise smooth, globally $H(\bcurl;\Omega)$ fields, see e.g.~\cite[Prop.~2.2.32]{ACL_2018}.
    As a result, only the line integrals over boundary curves remain in the sum, yielding the 
    right-hand side integral in \eqref{Stokes_S}.
\end{proof}

Let us further study the flux functionals defined on the surfaces $\Sigma_i$.

\begin{proposition} \label{prop:Phi}
    For each $1 \le i \le \beta_1$, the flux functional
    \begin{equation}
        \label{Phi}
        \Phi_i : C^\infty_c(\Omega) \to \RR, 
        \qquad 
        \bw \mapsto \int_{\Sigma_i} \bn \cdot \bw.
    \end{equation}
    extends to a continuous linear form on $H_0(\Div;\Omega)$
    which vanishes on $\bcurl H_0(\bcurl;\Omega)$: we have
    \begin{equation}
        \label{Phi_curl0}
        \Phi_i(\bcurl \bu) = 0, \qquad \bu \in H_0(\bcurl;\Omega).
    \end{equation}
\end{proposition}

\begin{proof}
    The argument is similar to that used in the proof of~\cite[Lemma~3.10]{ABDG_1998}, 
    although simpler since we consider fluxes defined on single surfaces.
    Using the fact that $\Sigma_i$ is Lipschitz and oriented, we can find a scalar function
    $\vp_i \in H^1(\omega_i)$ with $\omega_i = \Omega \setminus \Sigma_i$,
    whose trace vanishes on the positive side of $\Sigma_i$ (as pointed by its normal vector $\bn$) 
    and is 1 on its negative side. For $\bw \in C^\infty_c(\Omega)$, Green's formula yields then
    \begin{equation}
        \Phi_i(\bw) 
        = \int_{\Sigma_i} \bn \cdot \bw = \int_{\omega_i} \bw \cdot \bgrad \phi_i + \int_{\omega_i} (\Div \bw) \phi_i
    \end{equation}
    and in particular one has
    \begin{equation}
        \Phi_i(\bw) \le \norm{\bw}_{L^2(\omega_i)} \norm{\bgrad \phi_i}_{L^2(\omega_i)} + \norm{\Div \bw}_{L^2(\omega_i)} \norm{\phi_i}_{L^2(\omega_i)}
        \le \norm{\bw}_{H(\Div;\Omega)} \norm{\phi_i}_{H^1(\omega_i)}
    \end{equation}
    so that $\Phi_i$ admits indeed a continous extension in $H_0(\Div;\Omega)$.
    To verify \eqref{Phi_curl0} we then observe that for $\bu \in C^\infty_c(\Omega)$, 
    Stokes' formula \eqref{Stokes_S} gives 
    \begin{equation}
        \label{***}
        \Phi_i(\bcurl \bu) = \int_{\Sigma_i} \bn \cdot \bcurl \bu
        = \int_{\partial \Sigma_i} \btau \cdot \bu= 0
    \end{equation}
    where the last equality follows from the inclusion
    $\partial \Sigma_i \subset \partial \Omega$.
    We conclude by using the density of $C^\infty_c(\Omega)$ in
    $H_0(\bcurl;\Omega)$, the inclusion 
    $\bcurl H_0(\bcurl;\Omega) \subset H_0(\Div;\Omega)$,
    and the continuity of $\Phi_i$ in the latter space.
\end{proof}

\subsection{Construction by decomposition}
\label{sec:dec}

Our vector potentials are defined as a sum
\begin{equation}
    \label{A_dec}
    \bA_i := \bA_i^b + \bA_i^0, \qquad 1 \le i \le \beta_1,
\end{equation}
where the first terms lift some suitable boundary condition into the domain, and the second are correction terms with homogeneous boundary conditions.

Specifically, the $\bA_i^b \in H(\bcurl;\Omega)$ are piecewise smooth fields 
which lift the following boundary conditions:
\begin{equation}
    \label{Ab_bc}
    \bn_{\partial\Omega} \cdot \bcurl \bA_i^b = 0 ~~ \text{ on } \partial \Omega
    \qquad \text{ and } \qquad
    \int_{\partial \Sigma_j} \btau \cdot \bA_i^b = \delta_{i,j}, \qquad 1 \le i, j \le \beta_1
\end{equation}
where $\bn_{\partial\Omega}$ is the unit normal vector on the Lipschitz surface $\partial \Omega$, pointing outside of $\Omega$.
Notice that the first condition may be reformulated using a 
surface operator (see, e.g.,~\cite[Cor.~3.1.16]{ACL_2018}) 
so that it only involves boundary values of $\bA^b_i$ indeed.
We will say that the $\bA_i^b$ are \textit{lifted boundary potentials}.

The \textit{correction potentials} $\bA_i^0$ must then satisfy
\begin{equation}
    \label{A0_props}
    \bA_i^0 \in H_0(\bcurl;\Omega)
    \quad \text{ and } \quad 
    \bcurl \bcurl \bA_i^0 = - \bcurl \bcurl \bA_i^b.
\end{equation}
These properties are indeed sufficient to obtain a generating family of harmonic vector potentials.

\begin{proposition} \label{prop:w_basis}
    If \eqref{Ab_bc} and \eqref{A0_props} hold,
    then the potentials \eqref{A_dec} 
    satisfy 
    $\bcurl \bA_i \in \frH^1$ and 
    \begin{equation}
        \label{Phi_curlA}
        \Phi_j(\bcurl \bA_i) = \delta_{i,j}, \quad 1 \le i,j \le \beta_1.
    \end{equation}    
    In particular, the fields $\bw_i := \bcurl \bA_i$ form a basis of $\frH^1$.

\end{proposition}

\begin{proof}
    The equality $\Div \bw_i = 0$ is obvious, and $\bcurl \bw_i = 0$ is clear from \eqref{A0_props}.
    The boundary condition 
    $\bw_i \in H_0(\Div;\Omega)$ is also easily
    verified, indeed it is satisfied by 
    $\bcurl \bA_i^b$ thanks to \eqref{Ab_bc},
    and by $\bcurl \bA_i^0$ thanks to the inclusion 
    $\bcurl H_0(\bcurl;\Omega) \subset H_0(\Div;\Omega)$.    
    This shows that the fields $\bw_i$ are indeed tangent and harmonic.
    To show the relations \eqref{Phi_curlA} we first consider the 
    lifted boundary part. 
    Since it is piecewise smooth over the mesh $\cX_i$, 
    Stokes' formula \eqref{Stokes_S} yields 
    \begin{equation*}
        \Phi_j(\bcurl\bA_i^b)
        = \int_{\Sigma_j} \bn \cdot \bcurl \bA^b_i
        = \int_{\partial \Sigma_j} \btau \cdot \bA^b_i = \delta_{i,j}
    \end{equation*}
    where the last equality is \eqref{Ab_bc}.
    The desired relation \eqref{Phi_curlA} then follows from the 
    fact that the flux functionals $\Phi_j$ vanish for the 
    $\bA^0_i$ parts thanks to \eqref{Phi_curl0}.
    The linear independence of the fields $\bw_i$ readily follows, and the fact that they span $\frH^1$ is a consequence of de Rham's Theorem~\ref{th:dRL2}.
\end{proof}

We now proceed to construct the different terms of our decomposition.

\subsection{Lifted boundary potentials}
\label{sec:As}

We construct the lifted boundary potentials by pushforward:
we first define a smooth, compactly supported 
$\widehat{\bA}^b$ on the parametric domain $\cO$, 
which satisfies
\begin{equation} \label{hAb_flux}
    \hat{\bn} \cdot \bcurl \widehat{\bA}^b = 0 \quad \text{ on } r = 0,    
\end{equation}
where $\hat{\bn} = -\hat{\be}_r$ is the unit normal vector pointing
outside of $\cO$ on its $\{r=0\}$ boundary (see Figure~\ref{fig:mapping_T}), and 
\begin{equation} \label{hAb_circ}
    \int_{-1}^{1} \hat{\be}_z \cdot \widehat{\bA}^b(0,\theta,z) \rmd z = 1, \qquad \forall \theta \in \TT.
\end{equation}
Note that since $\widehat{\bA}^b$ is compactly supported in $\cO$, 
it vanishes on every face with the exception of the face $\{r = 0\}$ which is contained in $\cO$: it is on this face that the above boundary conditions are imposed.

Specifically, we may set 
\begin{equation} \label{hA}
\widehat{\bA}^b(r, \theta, z) := 
\begin{pmatrix} 0 \\ 0 \\ \hat{\alpha}(z)\hat{\vp}(r) \end{pmatrix}
\end{equation}
with $\hat{\alpha} \in C_c^\infty(-1,1)$ and 
$\hat{\vp} \in\cC_c^\infty[0,1)$ such that
    $\int_{-1}^{1}\hat{\alpha}(z) \rmd z = 1$
    and 
    $\hat{\vp}(0) = 1$.
One easily verifies that this guarantees \eqref{hAb_flux}
and \eqref{hAb_circ}.
This logical field is then pushed forward to define a piecewise smooth,
compactly supported potential on the domain $X_i = T_i(\cO)$,
which is then extended by zero outside of $X_i$. Thus, we set 
\begin{equation} \label{Abi}
    \bA_i^b(\bx) := 
    \begin{cases}
        \cF_i^1(\widehat{\bA}^b)(\bx)
        ~~ &\text{if} ~ \bx \in X_i
        \\
        0 ~~ &\text{else}
    \end{cases}
\end{equation}
where $\cF^1_i(\hat \bu) := ((DT_i)^{-T} \hat \bu) \circ T_i^{-1}$
is the covariant Piola transform (with $DT_i$ the Jacobian matrix of 
$T_i$) corresponding to the pushforward operator for 1-form proxy fields~\cite{hiptmair_actanum_2002}.

\begin{proposition} \label{prop:Ab}
    The field $\bA_i^b$ just defined satisfies the properties stated above:
    it is piecewise $C^\infty$ over the mesh $\cX_i$ in \eqref{Xia}, it belongs to $H(\bcurl; \Omega)$ 
    and the boundary conditions \eqref{Ab_bc} hold.
\end{proposition}

\begin{proof}
    By construction $\bA_i^b$ vanishes on $\overline X_{i,0}$, and its smoothness on $\overline X_{i,a}$, $a = 1, \dots m_i$, follows from the 
    smoothness of the logical potential \eqref{hA} and that of the mapping $T_i$ and its inverse on the subdomains 
    $\overline \cO_{i,a}$ and $\overline X_{i,a} = T_i(\overline \cO_{i,a})$. 
    To verify that $\bA_i^b \in H(\bcurl; \Omega)$, according to \cite[Prop.~2.2.32]{ACL_2018} we need to 
    show that (i) the tangential trace of the potential \eqref{Abi} 
    vanishes on $\partial X_i \setminus \partial \Omega$,
    and (ii) on every interface $f = \overline{X}_{i,a} \cap \overline{X}_{i,b}$ between two subdomains, and every vector $\btau_f$ tangent to $f$, 
    we have 
    \begin{equation}
        \label{tancontA}
        \btau_f \cdot \bA_i^b|_{X_{i,a}} 
        = \btau_f \cdot \bA_i^b|_{X_{i,b}} \quad \text{ on } f.
    \end{equation}
    The former property is clear since the field \eqref{hA}
    vanishes on a neighborhood of $\partial \cO \setminus \{r = 0 \}$,
    and the latter property follows from from the fact that a tangent 
    vector on $f$ is of the form 
    $$
    \btau_f(\bx) = (DT_i \hat \btau)(\hat \bx)
    $$
    where $\hat \btau$ is a vector tangent to the surface 
    $\hat f = T_i^{-1}(f)$, and where the Jacobian matrix of $T_i$ 
    can be evaluated from both subdomains $\cO_{i,a}$ or $\cO_{i,b}$
    since the directional derivative is taken along their common interface.
    The continuity of the tangential traces on $f$ then follows from
    the definition of the pushforward.

    We finally turn to the boundary conditions \eqref{Ab_bc}.
    Using the commutation relation
    $$
    \cF^2_i \bcurl \widehat{\bA^b} = \bcurl \cF^1_i  \widehat{\bA^b}
    $$
    satisfied by the 2-form pushforward operator $\cF^2_i$, see e.g. 
    \cite{hiptmair_actanum_2002}, and the fact that 
    surface integrals of smooth vector fields are 
    preserved by 2-form pushforwards \cite[Ch.~16.6]{lee_ISM_2012},
    we verify that the first boundary condition 
    $\bn_{\partial\Omega} \cdot \bcurl \bA_i^b = 0$ is satisfied 
    as a consequence of its parametric counterpart $\eqref{hAb_flux}$.
    To verify the second one, we remind that the intersection
    of $\partial \Sigma_j$ with $X_i$ 
    decomposes into $n(i,j)$ disjoint curves 
    according to \eqref{XGprime_alg}--\eqref{ind_duality}
    (or \eqref{XGprime} in the simpler case where $n(i,j) = \delta_{i,j}$)
    that is,
    $$
    X_i \cap \partial \Sigma_j =     
    \cup_{1 \le p \le n(i,j)} \tilde \Gamma(i,j,p)
    \quad \text{ with } ~ 
    \tilde \Gamma(i,j,p) = 
    T_i\big(\{0\}\times\{ \theta(i,j,p) \}\times(-1,1)\big)
    $$
    each curve being oriented by a unit vector induced 
    by the global orientation of $\Sigma_j$
    as defined by its normal vector field $\bn_{\Sigma_j}$.
    Using the fact that line integrals are preserved by 1-form pushforwards, we then write
    $$
    \int_{\partial \Sigma_j} \btau \cdot \bA_i^b 
    = 
    \sum_{p=1}^{n(i,j)} 
    \int_{\tilde \Gamma(i,j,p)} \btau \cdot \bA_i^b 
    =
    \sum_{p=1}^{n(i,j)} 
    \int_{-1}^{1} \hat \btau(i,j,p) \cdot \widehat{\bA_i^b}(0,\theta(i,j,p), z) \rmd z 
    $$
    where $\hat \btau(i,j,p)$
    is the tangent vector of the straight curve 
    $T_i^{-1}(\tilde \Gamma(i,j,p))$ 
    in the direction induced by that of $\tilde \Gamma(i,j,p)$,
    that is,
    $$
    \hat \btau(i,j,p) = {\rm ind}(i,j,p) \hat{\be}_z
    $$
    according to \eqref{indices_alg}. 
    Using \eqref{hAb_circ} and \eqref{ind_duality} then yields the desired result:
    $$
    \int_{\partial \Sigma_j} \btau \cdot \bA_i^b 
    = \sum_{p=1}^{n(i,j)} 
    {\rm ind}(i,j,p) 
    \int_{-1}^{1} \hat{\be}_z \cdot \widehat{\bA_i^b}(0,\theta(i,j,p), z) \rmd z 
    =
    \sum_{p=1}^{n(i,j)} 
    {\rm ind}(i,j,p) 
    = \delta_{i,j}.
    $$
\end{proof}

\subsection{Correction potentials}
\label{sec:A0}

To define the correction potentials we consider the following problem:
\medskip

Find $(\sigma_i,\bA_i^0,\bp_i)\in H_0^1(\Omega)\times H_0(\bcurl;\Omega)\times\frH^2$, such that 
\begin{equation}
    \label{A0_sys}
    \left\{\begin{aligned}
        \sprod{\sigma_i}{\tau} + \sprod{\bgrad\tau} {\bA_i^0} &= 0,\ &&\quad\forall\ \tau\in H_0^1(\Omega), 
        \\
        \sprod{\bgrad\sigma_i}{\bv} + \sprod{\bcurl\bA_i^0}{\bcurl\bv} + \sprod{\bp_i} {\bv} &= -\sprod{\bcurl\bA_i^b}{\bcurl\bv},\ &&\quad\forall\ \bv\in H_0(\bcurl;\Omega), 
        \\
        \sprod{\bA_i^0}{\bq} &= 0,\ &&\quad\forall\ \bq\in\frH^2.
    \end{aligned}\right.
\end{equation}

\begin{proposition} \label{prop:A0} 
    System \eqref{A0_sys} defines a unique field
    $\bA_i^0$ which satisfies \eqref{A0_props}.
\end{proposition}

\begin{proof}
The well-posedness follows by applying~\cite[Theorem 3.2]{AFW_2010}, 
to the adjoint (homogeneous) de Rham complex \eqref{dR_0}, which is possible thanks to the properties stated in~\cite[Sec.~4.2]{AFW_2010}.
By taking respectively 
$\bv = \bgrad\sigma_i$ and 
$\bv = \bp_i \in \frH^2$, we then observe that the solution satisfies
$\sigma_i=0$ and $\bp_i=0$.
The relation \eqref{A0_props} readily follows.
\end{proof}

\subsection{Main result and invariance principles}

By combining Propositions \ref{prop:w_basis}, \ref{prop:Ab} and \ref{prop:A0}, which hold under Assumption~\ref{as:Gamma},
we have proven our main result.
\begin{theorem} \label{th:H1_basis}
    The fields $\bA_i = \bA^b_i + \bA^0_i$
    constructed over the tunnel curves $\Gamma_i$, 
    $i = 1, \dots \beta_1$ as described above,
    are free vector potentials for the tangent harmonic fields.
    More precisely, the fields $\bw_i := \bcurl \bA_i$ 
    form a basis for the space $\frH^1$ 
    which is dual to the flux functionals $\Phi_i$ associated with the 
    reciprocal surfaces, in the sense that 
    \begin{equation}
        \Phi_j(\bcurl\bA_i)
        = \int_{\Sigma_j} \bn \cdot \bcurl \bA_i
        = \delta_{i,j}
    \end{equation}
    holds for all $1 \le i,j \le \beta_1$.
\end{theorem}

Our construction satisfies two (dual) invariance principles.
The first one is with respect to the tunnel curves:    
if $\Gamma_i$ and $\Gamma'_i$, $i = 1, \dots \beta_1$,
are two collections of piecewise smooth closed boundary curves 
satisfying the properties of Assumption~\ref{as:Gamma}
with the same reciprocal surfaces $\Sigma_i$, then 
their associated harmonic potentials are a priori different but their harmonic fields coincide
\begin{equation}
    \bcurl \bA_i = \bcurl \bA'_i, \quad i = 1, \dots \beta_1,
\end{equation}
since it follows from Theorem~\ref{th:H1_basis} that the 
fluxes $\Phi_i$ are unisolvent on the harmonic space $\frH^1$.

This invariance principle actually holds with respect to
the 1-chain \textit{homology classes}:
indeed if $\Gamma_i$ and $\Gamma'_i$ are two collections of piecewise smooth closed boundary curves that represent the same basis $[\Gamma_i]$ of the first homology group $\cH_1(\overline{\Omega})$, 
then we may consider the (unique) basis $[\Sigma_i]$ of the relative homology group $\cH_2(\overline{\Omega}, \partial \Omega)$ which is dual 
in the sense of the intersection pairing \eqref{asQ}, 
and any collection of piecewise smooth surfaces $\Sigma_i$ 
which represent this dual basis while intersecting nicely all the tunnel curves.
Because the intersection pairing coincides with the algebraic intersection 
number of any pair of representative cycles, we find that 
the surfaces $\Sigma_i$ are reciprocal to both curve collections,
so that the associated harmonic fields coincide again.

A dual invariance principle holds with respect to the reciprocal surfaces:
indeed if $\Sigma_i$ and $\Sigma'_i$, $i = 1, \dots \beta_1$,
are two collections of piecewise smooth surfaces
wich are reciprocal to the same tunnel curves $\Gamma_i$
in the sense of Assumption~\ref{as:Gamma}, then 
the harmonic fields $\bw_i = \bcurl \bA_i$ associated with these curves
have the same fluxes across both surface collections:
\begin{equation}
    \int_{\Sigma_i} \bn \cdot \bw_i 
    = \int_{\Sigma'_i} \bn \cdot \bw_i 
    , \quad i = 1, \dots \beta_1,
\end{equation}
as an immediate consequence of Theorem~\ref{th:H1_basis}.

Again, it follows from the properties of the intersection pairing \eqref{asQ} that this dual invariance principle actually holds with respect to the 
homology classes of 2-chains modulo the boundary.

\section{Structure-preserving discrete harmonic potentials}
\label{sec:discpot}

We now describe how the above construction can be implemented in a discrete setting to provide a parametrization of harmonic vector fields in terms of strong potentials associated with cavities and tunnels, in a way that also guarantees the correct dimensions for the harmonic spaces -- namely the Betti numbers of the domain $\Omega$. The approach proposed here applies to a large class of structure-preserving finite element spaces that share a certain number of properties listed in Section~\ref{sec:fem} below. 

Assuming that the domain $\Omega$ can be triangulated with a simplicial mesh, our construction applies for instance to Whitney~\cite{whitney_git_1957,bossavit_whitney_1988} and Nédélec~\cite{nedelec_mixed_1980} finite element spaces, and more generally high order FEEC spaces~\cite{AFW_2006}. It is however not restricted to these cases (where discrete harmonic spaces are known to have the correct dimensions \cite{AFW_2006}), indeed it also applies to finite element spaces where commuting projections can be constructed by interpolating geometric degrees of freedom, such as tensor-product spline spaces in single or multi-patch configurations~\cite{buffa_isogeometric_2011,buffa_multipatch_2019}.

\subsection{Discretization of the homogeneous and extended dual complexes}
\label{sec:fem}

Since harmonic fields can be described by potentials associated with differential operators from the dual complex extended to spaces without homogeneous boundary conditions, a natural choice is to consider a strong discretization of the dual de Rham complex \eqref{dR_0}, seen as a subcomplex of the one without boundary conditions (i.e., the primal one).

To better reflect this choice, and distinguish this \textit{extended dual complex} from the primal complex \eqref{dR},
let us rename the spaces with homogeneous boundary 
conditions as $\tilde V^\ell_0 := V^*_{3-\ell}$ 
and the associeted operators as $\tilde d^\ell_0 := d^*_{3-\ell}$.
We let then $\tilde V^\ell$ denote
the corresponding spaces without boundary conditions, 
equipped with their natural differential operators $\tilde d^\ell$.
Thus, we rewrite the dual complex \eqref{dR_0} as 
\begin{equation} \label{dR_tV0}
  0 
  \xleftarrow{~~} \tilde V^3_0 = L^2(\Omega)
  \xleftarrow{~ \tilde d_0^2 ~} \tilde V^2_0 = H_0(\Div; \Omega) 
  \xleftarrow{~ \tilde d_0^1 ~} \tilde V^1_0 = H_0(\bcurl;\Omega) 
  \xleftarrow{~ \tilde d_0^0 ~} \tilde V^0_0 = H^1_0(\Omega) 
  \xleftarrow{~~} 0
\end{equation}
and we see it as a subcomplex of 
\begin{equation} \label{dR_tV}
  0 
  \xleftarrow{~~} \tilde V^3 = L^2(\Omega)
  \xleftarrow{~ \tilde d^2 ~} \tilde V^2 = H(\Div; \Omega) 
  \xleftarrow{~ \tilde d^1 ~} \tilde V^1 = H(\bcurl;\Omega) 
  \xleftarrow{~ \tilde d^0 ~} \tilde V^0 = H^1(\Omega) 
  \xleftarrow{~~} 0.
\end{equation}
Notice that up to a change of signs in the grad and div operators, the latter coincides with the primal complex \eqref{dR}. For simplicity we may
disregard these sign changes and consider the plain differential operators
below. We distinguish, however, the primal and the extended dual complexes
which play different roles.

\subsubsection{Admissible discretizations}

We consider a discrete subcomplex of \eqref{dR_tV}, 
\begin{equation} \label{dR_tVh}
    0 
    \xleftarrow{~~} \tilde V^3_{h} 
    \xleftarrow{~ \Div ~} \tilde V^2_{h}
    \xleftarrow{~ \bcurl ~} \tilde V^1_{h}
    \xleftarrow{~ \bgrad ~} \tilde V^0_{h}
    \xleftarrow{~~} 0,
\end{equation}
associated with a mesh $\cM_h$ of $\Omega$
made of cells with Lipschitz and piecewise smooth boundaries as in Definition~\ref{def:pwCk}
(here the subscript $h$ will serve as an informal reference to the mesh).
More precisely, we consider that the finite-dimensional spaces are piecewise smooth,
\begin{equation} 
    \tilde V^\ell_h \subset C^\infty(\cM_h) \cap \tilde V^\ell,
    \qquad 0 \le \ell \le 3,
\end{equation}
assume that the following properties hold:
\begin{itemize}
    
    \item[(P1)]
    the sequence \eqref{dR_tVh} admits projection operators 
    \begin{equation}
        \label{Pi_h}
        \tilde \Pi^\ell: (C^\infty(\cM_h) \cap \tilde V^\ell) \to \tilde V^\ell_h, 
        \qquad 0 \le \ell \le 3,
    \end{equation}    
    which commute with the differential operators, i.e.,
    \begin{equation}
        \label{comm_Pi_h}
        \tilde d^\ell \tilde \Pi^\ell v = \tilde \Pi^{\ell+1} \tilde d^\ell v,
        \qquad 0 \le \ell \le 3,
    \end{equation}    

    \item[(P2)]
    the projection operators 
    preserve the homogenous boundary conditions, 
    in the sense that
    \begin{equation} \label{Pi_bc}
        \tilde \Pi^\ell(C^\infty(\cM_h) \cap \tilde V^\ell_{0}) 
        = \tilde V^\ell_h \cap \tilde V^\ell_{0},
    \end{equation}    

    \item[(P3)]
    the space $\tilde V^0_{h}$ contains $\beta_2$ functions $\psi_{h,i}$
    which are constant on the $\beta_2+1$ connected components of the boundary $\partial \Omega$, and more precisely
    \begin{equation}
        \label{psi_h_bc}
        \psi_{h,i}|_{S_j} = \delta_{i,j}, \qquad 1 \le i \le \beta_2, \quad 0 \le j \le \beta_2,
    \end{equation}
    
    \item[(P4)]     
    the projection onto $\tilde V^1_h$ preserves the line integrals over the reciprocal surface boundaries $\partial\Sigma_{i}$, $1 \le i \le \beta_1$
    in the sense that
    \begin{equation} \label{circ_Pia}
        \int_{\partial \Sigma_{i}} \btau \cdot \tilde \Pi^1 \ba
        = \int_{\partial \Sigma_{i}} \btau \cdot \ba,
        \qquad
        \ba \in C^\infty(\cM_h) \cap \tilde V^1,
    \end{equation}

    \item[(P5)] 
    the space $\tilde V^3_{h,0}$ contains at least one function with non-zero integral,

    \item[(P6)] 
        $\cM_h$ is not coarser than the meshes $\cX_i$ from \eqref{Xia}, in the sense that
    \begin{equation}
        C^\infty(\cX_i) \subseteq C^\infty(\cM_h), \qquad 1 \le i \le \beta_i,
    \end{equation}
    moreover $\cM_h$ is compatible with the surfaces $\Sigma_i$
    in the sense of Definition~\ref{def:compsurf}.

\end{itemize}

Using properties (P1) and (P2), we see that the same projection operators yield two commuting diagrams: 
one for the extended dual complex \eqref{dR_tV},
\begin{equation}
    \label{cd_tVh}
    \begin{tikzpicture}[ampersand replacement=\&, baseline] 
    \matrix (m) [matrix of math nodes,row sep=3em,column sep=1.4em,minimum width=1em] {
        ~~ 0 ~ \bbb
        \& ~~ C^\infty(\cM_h) \cap \tilde V^3 ~ \bbb
            \& ~~ C^\infty(\cM_h) \cap \tilde V^2  ~ \bbb
                \& ~~ C^\infty(\cM_h) \cap \tilde V^1  ~ \bbb
                    \& ~~ C^\infty(\cM_h) \cap \tilde V^0  ~ \bbb
                        \& ~~ 0 ~ \bbb
    \\
        ~~ 0 ~ \bbb
        \& ~~ \tilde V^3_h ~ \bbb
            \& ~~ \tilde V^2_h ~ \bbb
                \& ~~ \tilde V^1_h ~ \bbb
                    \& ~~ \tilde V^0_h ~ . \bbb
                        \& ~~ 0 ~ \bbb
    \\
    };
    \path[-stealth]
    (m-1-2) edge  (m-1-1)
    (m-1-3) edge node [above] {$\Div$}   (m-1-2)
    (m-1-4) edge node [above] {$\bcurl$} (m-1-3)
    (m-1-5) edge node [above] {$\bgrad$} (m-1-4)
    (m-1-6) edge  (m-1-5)
    (m-2-2) edge  (m-2-1)
    (m-2-3) edge node [above] {$\Div$}   (m-2-2)
    (m-2-4) edge node [above] {$\bcurl$}  (m-2-3)
    (m-2-5) edge node [above] {$\bgrad$} (m-2-4)
    (m-2-6) edge    (m-2-5)
    (m-1-1) edge      (m-2-1)
    (m-1-2) edge node [right] {$~\tilde \Pi^3$} (m-2-2)
    (m-1-3) edge node [right] {$~\tilde \Pi^2$} (m-2-3)
    (m-1-4) edge node [right] {$~\tilde \Pi^1$} (m-2-4)
    (m-1-5) edge node [right] {$~\tilde \Pi^0$} (m-2-5)
    (m-1-6) edge      (m-2-6)
    ;
    \end{tikzpicture}
\end{equation}
and another one for the (homogeneous) dual complex \eqref{dR_tV0} 
\begin{equation}
    \label{cd_tVh0}
    \begin{tikzpicture}[ampersand replacement=\&, baseline] 
    \matrix (m) [matrix of math nodes,row sep=3em,column sep=1.4em,minimum width=1em] {
        ~~ 0 ~ \bbb
        \& ~~ C^\infty(\cM_h) \cap \tilde V^3_0 ~ \bbb
            \& ~~ C^\infty(\cM_h) \cap \tilde V^2_0  ~ \bbb
                \& ~~ C^\infty(\cM_h) \cap \tilde V^1_0  ~ \bbb
                    \& ~~ C^\infty(\cM_h) \cap \tilde V^0_0  ~ \bbb
                        \& ~~ 0 ~ \bbb
    \\
        ~~ 0 ~ \bbb
        \& ~~ \tilde V^3_{h,0} ~ \bbb
            \& ~~ \tilde V^2_{h,0} ~ \bbb
                \& ~~ \tilde V^1_{h,0} ~ \bbb
                    \& ~~ \tilde V^0_{h,0} ~ . \bbb
                        \& ~~ 0 ~ \bbb
    \\
    };
    \path[-stealth]
    (m-1-2) edge  (m-1-1)
    (m-1-3) edge node [above] {$\Div$}   (m-1-2)
    (m-1-4) edge node [above] {$\bcurl$} (m-1-3)
    (m-1-5) edge node [above] {$\bgrad$} (m-1-4)
    (m-1-6) edge  (m-1-5)
    (m-2-2) edge  (m-2-1)
    (m-2-3) edge node [above] {$\Div$}   (m-2-2)
    (m-2-4) edge node [above] {$\bcurl$}  (m-2-3)
    (m-2-5) edge node [above] {$\bgrad$} (m-2-4)
    (m-2-6) edge    (m-2-5)
    (m-1-1) edge      (m-2-1)
    (m-1-2) edge node [right] {$~\tilde \Pi^3$} (m-2-2)
    (m-1-3) edge node [right] {$~\tilde \Pi^2$} (m-2-3)
    (m-1-4) edge node [right] {$~\tilde \Pi^1$} (m-2-4)
    (m-1-5) edge node [right] {$~\tilde \Pi^0$} (m-2-5)
    (m-1-6) edge      (m-2-6)
    ;
    \end{tikzpicture}
\end{equation}
where the discrete spaces with homogeneous boundary conditions are defined as
\begin{equation}
    \tilde V^\ell_{h,0} := \tilde V^\ell_h \cap \tilde V^\ell_0, \qquad 0 \le \ell \le 3.
\end{equation}

\subsubsection{Practical examples}

In practice, the above properties can be achieved with standard finite element spaces such as Whitney, Nédélec or higher order FEEC spaces \cite{AFW_2006}.
More generally, one may use finite element spaces associated with degrees of freedom 
\begin{equation} \label{dofs}
    \tilde \sigma^\ell_i: (C^\infty(\cM_h) \cap \tilde V^\ell) \to \RR, 
    \qquad 1 \le i \le \tilde N_\ell := \dim(\tilde V^\ell_h),
\end{equation}
satisfying relations of the form
\begin{equation} \label{dof_compat}
    \tilde \sigma^{\ell+1}_i(\tilde d^\ell v) 
        = \sum_{j=1}^{\tilde N_{\ell}} \tilde{\mathbb{D}}^\ell_{ij} \tilde \sigma^{\ell}_j(v)
        , \qquad 1 \le i \le \tilde N_{\ell+1}, 
\end{equation}
for some matrices $\tilde{\mathbb{D}}^\ell$ which represent the differential operators $\tilde d^\ell$ at the discrete level.
It is indeed easy to see that the relations \eqref{dof_compat} imply that the associated finite element interpolation operators, defined by 
\begin{equation}
    \label{Pi_h_sigma}
    \tilde \sigma^\ell_i(\tilde \Pi^\ell v) 
    = \tilde \sigma^\ell_i(v),
    \qquad 1 \le i \le \tilde N_\ell,
\end{equation}
satisfy the commuting property \eqref{comm_Pi_h}.

In the case of multi-patch spline finite elements for instance \cite{buffa_isogeometric_2011,buffa_multipatch_2019}, 
geometric degrees of freedom of point, edge, face and volume types \cite{bochev_principles_2006} can be used inside the patches as well as on the interfaces between matching patches \cite{buffa_approximation_2015, guclu_broken_2023} and on the domain boundaries.
Such degrees of freedom satisfy a relation of the type \eqref{dof_compat} thanks to Stokes' formula and they are compatible with homogeneous boundary conditions, i.e., they satisfy (P1) and (P2).

As for Property (P4), it will also be satisfied if edge integral degrees of freedom are used on the domain boundaries for the space of discrete 1-forms $\tilde V^1_h$, and if the curves $\Gamma_i$ consist of (boundary) edges of the mesh. In practice, algorithms have been proposed to compute such discrete tunnel curves in the cases where a parametrization is not already available, see e.g. \cite{Hiptmair_Ostrowski_2002,Hiptmair_Ostrowski_2021,pitassi_generators_2025}. Note that an explicit description of the reciprocal surfaces is not required, as these are not involved in the computation of the harmonic potentials.

Finally we observe that Properties (P3) and (P5) are easy to satisfy by standard finite elements, and that (P6)
is not very restrictive since in most cases one can subdivide the surfaces $\Sigma_i$ into smaller pieces to make them compatible with a fine mesh.

\subsection{Discrete harmonic spaces}
\label{sec:discharm}

A discretization of the continuous harmonic spaces \eqref{frH} is 
classically provided by the harmonic spaces of the discrete complex~\cite{AFW_2006}. Since the primal and dual complexes have the same harmonic spaces at the continuous level, one may use the dual discrete complex \eqref{cd_tVh0}: this leads to the discrete vector-valued spaces
\begin{equation}
    \label{tilde_Hh}
    \tilde \frH^1_{h,0} 
        := \tilde \frZ^1_{h,0} \cap (\bgrad \tilde V^0_{h,0})^\perp,
        \qquad 
    \tilde \frH^2_{h,0} 
        := \tilde \frZ^2_{h,0} \cap (\bcurl \tilde V^1_{h,0})^\perp,
\end{equation}
where the discrete kernels are defined as usual,
\begin{equation}
    \label{tilde_Zh}
    \tilde \frZ^1_{h,0}
    := \{ \bv \in \tilde V^1_{h,0} : \bcurl \bv = 0 \},
    \qquad
    \tilde \frZ^2_{h,0}
    := \{ \bw \in \tilde V^2_{h,0} : \Div \bw = 0 \},
\end{equation}
and we may define scalar harmonic spaces in a similar way (using the fact that the end operators in \eqref{cd_tVh0} are trivial), as
\begin{equation}
    \label{tilde_Hh_scal}
    \tilde \frH^0_{h,0} 
        := \tilde \frZ^0_{h,0} = \{ \varphi \in \tilde V^0_{h,0} : \bgrad \varphi = 0 \},
        \qquad 
    \tilde \frH^3_{h,0} 
        := \tilde V^3_{h,0} \cap (\Div \tilde V^2_{h,0})^\perp.
\end{equation}

According to the relabeling $V^*_{\ell} = \tilde V^{3-\ell}_0$, the above spaces discretize the harmonic spaces $\frH^*_\ell = \frH^\ell$, so that we may also denote them as
\begin{equation}
    \label{Hh}
    \frH^\ell_{h} := 
    \tilde \frH^{3-\ell}_{h,0}.
\end{equation}

\begin{remark}
    The choice made here of discretizing the tangent harmonic spaces
    $\frH^1_{h}$ with 2-forms is consistent with the use of (strong) vector potentials discretized as 1-forms.
    A different choice is made in \cite{AR_construction_2013}, where the 
    harmonic fields of $\frH^1$ are represented 
    by discrete \textit{loop fields}, which are discrete 1-forms with nonzero circulations along 1-cycles that correspond to our tunnel curves.
\end{remark}

Because they commute with the differential operators, the projection operators \eqref{cd_tVh0} induce mappings 
between the respective cohomology groups, 
\begin{equation}
    \label{cohmap}
    [\tilde \Pi^\ell]: \quad
    \tilde \frZ^{\ell}_0/\tilde \frB^{\ell}_0
    ~ \to ~
    \tilde \frZ^{\ell}_{h,0}/\tilde \frZ^{\ell}_{h,0},
    \quad 
    [v] \mapsto [\tilde \Pi^\ell v]
\end{equation}
and one easily verifies that this mapping is surjective (see for instance \cite[Sec.~2.2]{AFW_2006}).
As the former group coincides with $\frZ^*_{3-\ell}/\frB^*_{3-\ell}$ it is isomorphic to the harmonic space 
$\frH^*_{3-\ell} = \frH^{3-\ell}$, and the latter is isomorphic to 
$\tilde \frH^{\ell}_{h,0} = \frH^{3-\ell}_{h}$: in particular this yields 
(using Theorem~\ref{th:dRL2})
\begin{equation} \label{dimHh_up}
    \dim(\frH^\ell_{h}) \le \dim(\frH^\ell) = \beta_\ell, \qquad 0 \le \ell \le 3.
\end{equation}

\begin{remark} \label{rem:dimHh}
    The commutation property alone does not guarantee that the discrete harmonic spaces have the correct dimensions: for instance, projecting the trivial complex $0 \to \RR^2 \to 0$ onto 
    $0 \to \RR \to 0$ reduces the dimension of the homology space from two to one.
\end{remark}

Let us conclude this section by verifying that discrete potentials may a priori be used for the 
discrete harmonic fields. 
\begin{proposition}
The following inclusions hold under the above assumptions:
\begin{equation}
    \label{frH2h_sub_grad}
    \tilde \frH^1_{h,0} \subseteq \big(\tilde V^1_h \cap H_0(\bcurl 0;\Omega)\big) \subseteq \bgrad \tilde V^0_h,
\end{equation}
and 
\begin{equation}
    \label{frH1h_sub_curl}
    \tilde \frH^2_{h,0} \subseteq \big(\tilde V^2_h \cap H_0(\Div 0;\Omega)\big) \subseteq \bcurl \tilde V^1_h.
\end{equation}
\end{proposition}

\begin{proof}
    Let $\bv_h \in \tilde \frH^1_{h,0}$. We have 
    $\bv_h \in H_0(\bcurl 0;\Omega)$ by construction 
    and the argument used to show \eqref{H0c0_sub_grad}
    yields $\bv_h = \bgrad \phi$ for some 
    $\phi \in H^1(\Omega)$. Since $\bv_h$ is piecewise smooth (and bounded), we observe that $\phi$ is also piecewise smooth, in particular we may apply the commuting projecion operators \eqref{cd_tVh}. This yields 
    \begin{equation*}
    \bv_h = \tilde \Pi^1 \bv_h = \tilde \Pi^1 \bgrad \phi = \bgrad \tilde \Pi^0 \phi
    \end{equation*}
    which shows \eqref{frH2h_sub_grad}. 
    The embeddings \eqref{frH1h_sub_curl} are proven in a similar way:
    For $\bw_h \in \tilde \frH^2_{h,0}$ we have again
    $\bw_h \in H_0(\Div 0;\Omega)$ by construction 
    and the argument used to show \eqref{H0d0_sub_curl}
    yields $\bw_h = \bcurl \bA$ for a vector potential 
    $\bA \in H(\bcurl;\Omega)$ defined as the restriction of some $\bar \bA$
    on $\Omega$. Here we may consider an explicit potential 
    defined on a cube $C_R = (-R,R)^3$ containing $\overline \Omega$, 
    \begin{equation*}
    \bar \bA(\bx) = \Big(0, 
        \int_{-R}^{x_1} \bar w_{h,3}(x', x_2, x_3) \rmd x', 
        -\int_{-R}^{x_1} \bar w_{h,2}(x', x_2, x_3) \rmd x' \Big),
    \end{equation*}
    where $\bar \bw_h$ is defined as the extension of $\bw_h$ by 0 on 
    $C_R \setminus \Omega$. 
    From $\bw_h \in C^\infty(\cM_h)$ we infer that 
    $\bA \in C^\infty(\cM_h) \cap H(\bcurl;\Omega)$, so that we may again apply the 
    commuting projection operators \eqref{cd_tVh}. This yields
    \begin{equation*}
    \bw_h = \tilde \Pi^2 \bw_h = \tilde \Pi^2 \bcurl \bA = \bcurl \tilde \Pi^1 \bA
    \end{equation*}
    which proves \eqref{frH1h_sub_curl}.
\end{proof}

\subsection{Discrete harmonic scalar potentials for $\frH^2_h = \tilde{\frH}^1_{h,0}$}

To discretize the harmonic scalar potentials 
from Section~\ref{sec:scapot}, we perform a simple Galerkin projection
on the discrete subspace of \eqref{H1S} obtained by supplementing the homogeneous 
space $\tilde V^0_{h,0}$ with the functions from Property (P3), namely
\begin{equation} \label{tV0hc}
\tilde V^0_{h,S} := 
\tilde V^0_{h,0} \oplus 
        \bigoplus_{i = 1}^{\beta_2} \Span(\psi_{h,i})
        \subset
    \tilde V^0_{h} \cap H^1_S(\Omega).
\end{equation}
For each $1 \le i \le \beta_2$, we thus let $\phi_{h,i} \in \tilde V^0_{h,S}$ 
solve the discrete Poisson problem
\begin{equation}
    \label{phi0h_eq}
    \sprod{\bgrad \vp}{\bgrad \phi_{h,i}} = \vp|_{S_i}
        \qquad \forall \vp \in \tilde V^0_{h,S}.
\end{equation}

\begin{theorem} \label{th:H2h_basis}
    The fields $\bv_{h,i} := \bgrad \phi_{h,i}$,
    $i = 1, \dots \beta_2$, form a basis of $\frH^2_h = \tilde \frH^1_{h,0}$.
    Moreover, they satisfy
    \begin{equation}
        \label{prods_vi}
        \sprod{\bgrad \psi_{h,j}}{\bv_{h,i}} = \delta_{i,j}, 
        \quad 1 \le i,j \le \beta_2.
    \end{equation}
\end{theorem}

\begin{proof} 
Problem~\eqref{phi0h_eq} is well-posed as a Galerkin projection of the coercive problem 
\eqref{phi_i_pbm}.
One then easily verifies that the fields $\bv_{h,i} := \bgrad \phi_{h,i}$ 
all belong to $\tilde \frH^1_{h,0}$ 
by using the same arguments as in Proposition~\ref{prop:gradbasis}.
To see that they are linearly independent, and hence 
that they span $\tilde \frH^2_{h}$, it suffices to 
use $\vp = \psi_{h,j}$ in \eqref{phi0h_eq}. This also shows \eqref{prods_vi}.
\end{proof}

\subsection{Discrete harmonic vector potentials for $\frH^1_h = \tilde{\frH}^2_{h,0}$}
\label{sec:construct_h}

The construction of discrete vector potentials follows the same steps
as in the continuous case: we define them as a sum
\begin{equation}
    \label{Ah_dec}
    \bA_{h,i} := \bA^b_{h,i} + \bA^0_{h,i} \in \tilde V^1_h,  \qquad 1 \le i \le \beta_1,
\end{equation}
where the first term is a \textit{discrete lifted boundary potential} that carries an inhomogeneous boundary condition associated with the tunnel curve $\Gamma_i$, and the second term is a correction that enforces the curl-free property for the resulting field $\bcurl \bA_{h,i}$ at the discrete level.

\paragraph{Discrete lifted boundary potentials}

The fields $\bA^b_{h,i} \in \tilde V^1_h$ 
must satisfy the same boundary properties as in \eqref{Ab_bc}, namely
\begin{equation}
    \label{Abh_bc}
    \bn_{\partial \Omega} \times \bcurl \bA^b_{h,i} = 0
    \qquad \text{ and } \qquad
    \int_{\partial \Sigma_j} \btau \cdot \bA^b_{h,i} = \delta_{i,j}, \qquad 1 \le i,j \le \beta_1.
\end{equation}
Note that these expressions are well-defined in $\tilde V^1_h$ according to the inclusion $\bcurl \tilde V^1_h \subset H(\Div;\Omega)$ and 
the compatibility of the surfaces and the mesh.

In practice there are several ways to construct discrete lifted boundary potentials:
\begin{itemize}

    \item a general method is to interpolate the potential \eqref{Abi} using the projection operator \eqref{Pi_h}, 
    \begin{equation}
        \label{Abhi}
        \bA^b_{h,i} := \tilde \Pi^1 \bA^b_i.
    \end{equation}

    \item if the boundary degrees of freedom in $\tilde V^1_h$ 
    correspond to edge integrals lying on the boundary mesh, and if both the tunnel curves $\Gamma_i$ and the reciprocal curves $\partial \Sigma_j$ consist of edges of the same boundary mesh,
    then one can build a valid $\bA^b_{h,i}$ by setting its degrees of freedom to 1 on the edges that point into
    the tunnel curve $\Gamma_i$, and to 0 on the other edges, as in Fig.~\ref{fig:collar_field}.
    
\end{itemize}

Observe that the second method does not require any knowledge of the continuous potential $\bA^b_i$ defined in Section~\ref{sec:As}. Or better, it directly defines a boundary potential $\bA^b_{i}$ which may be less regular than the one defined in Section~\ref{sec:As}, but is regular enough to be used in our construction since it belongs to $H(\bcurl;\Omega)$ and satisfies the boundary conditions \eqref{Ab_bc}.

In both cases, it is easy to see that the resulting fields $\bA^b_{h,i}$ satisfy the boundary conditions \eqref{Abh_bc}:
In the case of the interpolation \eqref{Abhi}, one first uses the commutation property \eqref{cd_tVh} to write 
\begin{equation*}
    \bn_{\partial \Omega} \times \bcurl \bA^b_{h,i}
    = \bn_{\partial \Omega} \times \bcurl \tilde \Pi^1 \bA^b_{i}
    = \bn_{\partial \Omega} \times \tilde \Pi^2 \bcurl \bA^b_{i} 
    = 0
\end{equation*}
where the last equality follows from \eqref{Ab_bc} and \eqref{Pi_bc} with $\ell = 2$. 
Finally the line integrals along the boundaries $\partial \Sigma_j$ are preserved by $\tilde \Pi^1$ as seen in \eqref{circ_Pia},
so that \eqref{Abh_bc} follows from \eqref{Ab_bc}.
With the direct construction, one infers from Stokes' formula that the flux of $\bcurl \bA^b_{i}$ vanishes across any face of the mesh, as this flux always amounts to two unit edge integrals with oposite signs, and the line integrals along the boundaries $\partial \Sigma_j$ are equal to 1 for $j=i$ and 0 otherwise, according to the intersection duality property.

\begin{figure}[!h]
    \begin{center}
    \def\svgwidth{200pt}
    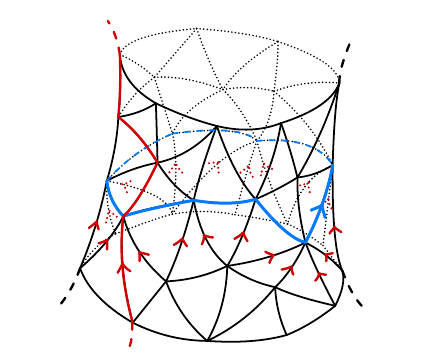
    \end{center}
\caption{%
    Fully discrete construction: if the boundary $\partial \Omega$ is described by a surface mesh (possibly curvilinear) where the tangent components of the fields in $\tilde V^1_h$ are determined by their edge integrals, a valid lifted boundary potential $\bA^b_{h,i}$ can be obtained by setting its edge degrees of freedom to 1 on the edges (marked by red arrows) pointing to a discrete tunnel curve $\Gamma_i$ (in blue) and to 0 on the other edges: the resulting fields (called \textit{collar fields} in \cite{Hiptmair_Ostrowski_2021}) indeed satisfy the boundary conditions \eqref{Abh_bc} thanks to Stokes' formula and the intersection duality between the tunnel curves and their reciprocal surfaces.
}
\label{fig:collar_field}
\end{figure}

\paragraph{Discrete correction potentials}

Once lifted boundary potentials $\bA^b_{h,i}$ are constructed, correction potentials 
can be obtained by solving a standard discretization of \eqref{A0_sys}, namely

Find $(\sigma_{h,i},\bA^0_{h,i},\bp_{h,i})\in \tilde V^0_{h,0} \times \tilde V^1_{h,0} \times \tilde \frH^1_{h,0}$, such that 
\begin{equation}
    \label{A0h_sys}
    \left\{\begin{aligned}
        \sprod{\sigma_{h,i}}{\tau} + \sprod{\bgrad\tau}{ \bA_{h,i}^0} &= 0,\ 
        &&\quad\forall\ \tau \in \tilde V^0_{h,0}, 
        \\
        \sprod{\bgrad\sigma_{h,i}}{\bv} + \sprod{\bcurl\bA_{h,i}^0}{ \bcurl\bv} 
        + \sprod{\bp_{h,i}}{ \bv} &= -\sprod{\bcurl\bA_{h,i}^b}{\bcurl\bv},\ 
        &&\quad\forall\ \bv\in \tilde V^1_{h,0}, 
        \\
        \sprod{\bA_{h,i}^0}{\bq} &= 0,\ &&\quad\forall\ \bq\in \tilde \frH^1_{h,0}.
    \end{aligned}\right.
\end{equation}
This problem admits a unique solution according to the existence of commuting projections. 
Indeed for any discrete sequence this implies a discrete Poincaré inequality,
and in turn an inf-sup condition according to the analysis of \cite{AFW_2006}.
(Note that these discrete Poincaré inequalities or inf-sup conditions may not be stable with respect to the discretization parameters,
as we did not assume the existence of projections that are uniformly bounded in the $\tilde V^\ell$ or $L^2$ norms.)
 
We next observe that the flux functionals $\Phi_i$ from \eqref{Phi} are well-defined on $\tilde \frH^2_{h,0}$ since the latter is a subspace of $\tilde V^2_{0} = H_0(\Div;\Omega)$. We then have the following result.

\begin{theorem} \label{th:H1h_basis}
    The fields $\bw_{h,i} := \bcurl \bA_{h,i}$,
    $i = 1, \dots \beta_1$, form a basis of $\frH^1_h = \tilde \frH^2_{h,0}$.
    Moreover, their fluxes on the surfaces $\Sigma_j$ 
    satisfy
    \begin{equation}
        \label{Phi_curlAh}
        \Phi_j(\bw_{h,i}) = \int_{\Sigma_{j}} \bn \cdot \bw_{h,i} = \delta_{i,j}, \qquad 1 \le i,j \le \beta_1.
    \end{equation}    
\end{theorem}

\begin{proof}
From \eqref{Ah_dec}, 
it is clear that $\bw_{h,i} = \bcurl\bA_{h,i} \in \tilde V^2_h$ 
satisfies $\Div \bw_{h,i} = 0$.
As for the boundary condition $\bn \cdot \bw_{h,i} = 0$ on $\partial \Omega$, it is easily verified by observing that 
both $\bcurl \bA^b_{h,i}$ and $\bcurl \bA^0_{h,i}$ are in $H_0(\Div;\Omega)$:
the first one thanks to \eqref{Abh_bc}, and the second one as the curl of a field in 
$\tilde V^1_0 = H_0(\bcurl;\Omega)$.

Thus, the fields $\bw_{h,i}$ are in $\tilde \frZ^2_{h,0}$.
To verify that they are in $\tilde \frH^2_{h,0}$ we must show that they also belong to $(\bcurl \tilde V^1_{h,0})^\perp$,
and for this we use the second equation from \eqref{A0h_sys}: taking first $\bv = \bgrad\sigma_{h,i}$ shows that $\sigma_{h,i} = 0$, then $\bv = \bp_{h,i}$ yields $\bp_{h,i} = 0$, so that one finally obtains
\begin{equation} \label{A0h_cc}
    \sprod{\bw_{h,i}}{\bcurl\bv} = \sprod{\bcurl(\bA^0_{h,i}+\bA^b_{h,i})}{\bcurl\bv} = 0,
        \quad\forall\ \bv\in \tilde V^1_{h,0}
\end{equation}
which shows that $\bw_{h,i} \in \tilde \frH^2_{h,0}$ indeed.

Finally, to evaluate the fluxes across a surface $\Sigma_{j}$ we use 
the fact that $\bA^0_{h,i} \in H_0(\bcurl;\Omega)$, which yields
$\Phi_j(\bcurl \bA^0_{h,i}) = 0$ according to Proposition~\ref{prop:Phi}:
this gives
\begin{equation*}
        \Phi_j(\bw_{h,i}) = 
        \Phi_j(\bcurl \bA^b_{h,i})
        = \int_{\Sigma_{j}} \bn \cdot \bcurl \bA^b_{h,i}
        = \int_{\partial \Sigma_{j}} \btau \cdot \bA^b_{h,i} 
        = \delta_{i,j}
\end{equation*}
where we have used Stokes' formula \eqref{Stokes_S} and the second boundary property \eqref{Abh_bc}.
This shows \eqref{Phi_curlAh} and the linear independence of the fields $\bw_{h,i}$.
In particular, this implies that $\dim(\tilde \frH^2_{h,0}) \ge \beta_1$,
and the spanning property follows from \eqref{dimHh_up}.
\end{proof}

\begin{remark} \label{rem:A0h_perpZ}
    In the proof above, we further note that the first and last equations from \eqref{A0h_sys} yield
    \begin{equation} \label{A0h_perpZ}
    \bA_{h,i}^0 
    \in (\bgrad \tilde V^0_{h,0} \poplus \tilde \frH^1_{h,0})^\perp
    = (\tilde \frZ^1_{h,0})^\perp
    \end{equation}
where we remind that $\tilde \frZ^1_{h,0}$ 
is the kernel of the curl in $\tilde V^1_{h,0}$,
see \eqref{tilde_Hh}--\eqref{tilde_Zh}.
\end{remark}

We conclude this section by observing that our construction yields a direct proof of cohomological correctness.
\begin{theorem}
    Under the assumptions listed in Section~\ref{sec:fem},
    the discrete harmonic spaces $\frH^\ell_h$ satisfy
    \begin{equation*}
        \dim(\frH^\ell_h) = \dim(\tilde \frH^{3-\ell}_{h,0}) = \dim(\frH^\ell) = \beta_\ell, \qquad 0 \le \ell \le 3.
    \end{equation*}
\end{theorem}

\begin{proof}
    The first equalities follow from the definition \eqref{Hh}, and the third ones have been established in Theorem~\ref{th:dRL2}. As for the second ones, 
    for $\ell = 1$ and 2 they have been proven in 
    Theorem~\ref{th:H1h_basis} and \ref{th:H2h_basis} respectively.
    For $\ell = 3$ the equality follows from 
    the upper bound \eqref{dimHh_up} since $\beta_3 = 0$
    (note that $\dim(\tilde \frH^0_{h,0}) = 0$ is also clear from 
    \eqref{tilde_Hh_scal} since $\tilde V^0_{h,0} \subset H^1_0(\Omega)$).
    Finally for $\ell = 0$ the bound \eqref{dimHh_up} yields 
    $\dim(\tilde \frH^3_{h,0}) = \dim(\frH^0_h) \le \beta_0 = 1$ and we observe
    that the existence of a function in $\tilde V^3_{h,0}$ with non-zero integral, 
    assumed in Section~\ref{sec:fem},
    guarantees using Stokes' formula that $\Div \tilde V^2_{h,0}$ is a proper subspace of $\tilde V^3_{h,0}$, hence the definition \eqref{tilde_Hh_scal} yields $\dim(\frH^0_h) = \dim(\tilde \frH^3_{h,0}) \ge 1$. The equality follows.
\end{proof}

\subsection{Computational Illustration}

To complement the theoretical construction presented in the preceding sections, we provide a computational illustration of our method.
The implementation was carried out using PSYDAC \cite{Guclu2022}, an academic, open-source, high-performance Python library designed for isogeometric analysis that offers access to high order FEEC spaces.
The computational domain is that of a magnetohydrodynamic plasma equilibrium obtained via the DESC code \cite{10.1063/5.0020743}; specifically, $\Omega$ is the volume between two magnetic flux surfaces, which results in a cavity along the magnetic axis, endowing the domain with a geometry similar to that of a hollow torus, see Figure \ref{fig:hollow_torus_domain}. 
\begin{figure}[ht!]
    \centering
    \includegraphics[width=.5\linewidth]{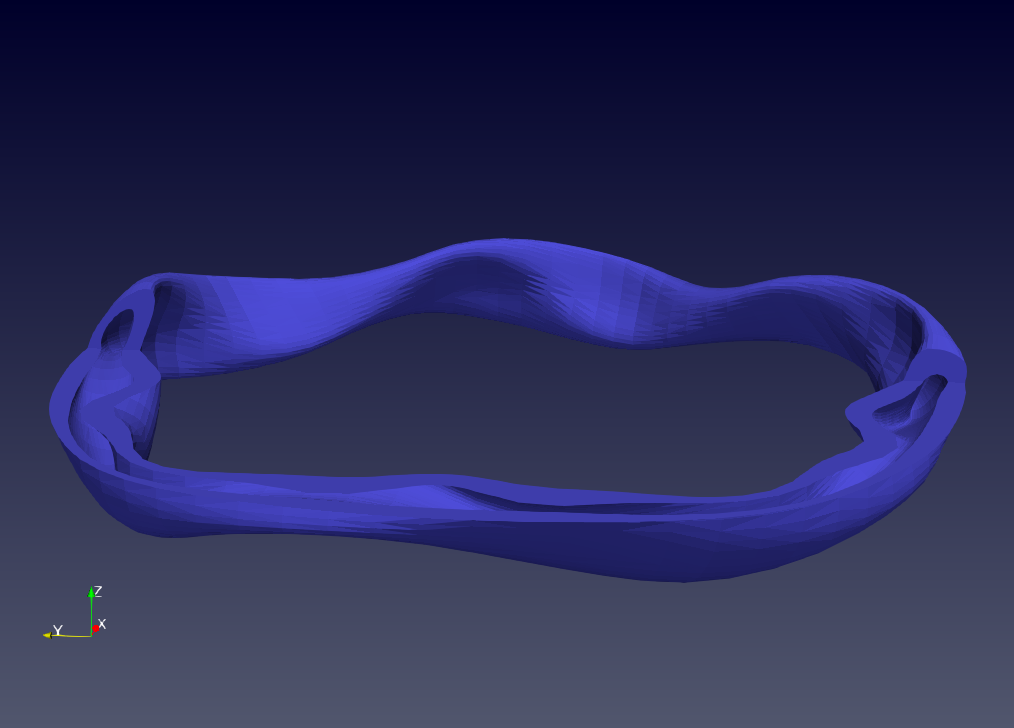}
    \caption{A domain diffeomorphic to a hollow torus (with one portion removed for visibility) featuring two tunnels, one of which is also a cavity: the inner tunnel may be called poloidal since it is characterized by closed poloidal curves, and the outer tunnel may be called toroidal since it is characterized by closed toroidal curves.}
    \label{fig:hollow_torus_domain}
\end{figure}
This cavity is at the same time a topological tunnel; hence the domain has two tunnels: one characterized by \textit{poloidal} curves around the ``inner'' tunnel (the cavity) and the other by \textit{toroidal} curves around the ``outer'' tunnel. These tunnels generate two tangent harmonic fields, for which we will compute vector potentials.

Here the domain is described by a parametrization associated with a logical domain that is a hollow cylinder. In PSYDAC, the spatial discretization is performed using tensor-product B-splines on the logical domain, and through the use of pushforward operators applied to these logical basis functions.
With the given parametrization, it is straightforward to identify the spline coefficients that allow to build the collar fields which we use as lifted boundary potentials $\bA^b_{h,i}$ according to the discrete construction described in Section~\ref{sec:construct_h}. These lifted boundary potentials are illustrated in 
Figure~\ref{fig:lifting_fields}. 
The first lifted boundary potential is called poloidal since it is built around one arbitrary poloidal tunnel curve. It is supported in the neighborhood of that curve and points in toroidal direction.
The second lifted boundary potential is called toroidal. It is supported in the neighborhood of one arbitrary toroidal tunnel curve and points in the poloidal direction.
In more general domains we note that the computation of the lifted boundary potentials may prove more challenging and refer to \cite{Hiptmair_Ostrowski_2002,Hiptmair_Ostrowski_2021} for algorithms to compute discrete tunnel curves and the corresponding collar fields.
\begin{figure}[ht!]%
    \centering
    \includegraphics[width=.47\textwidth]{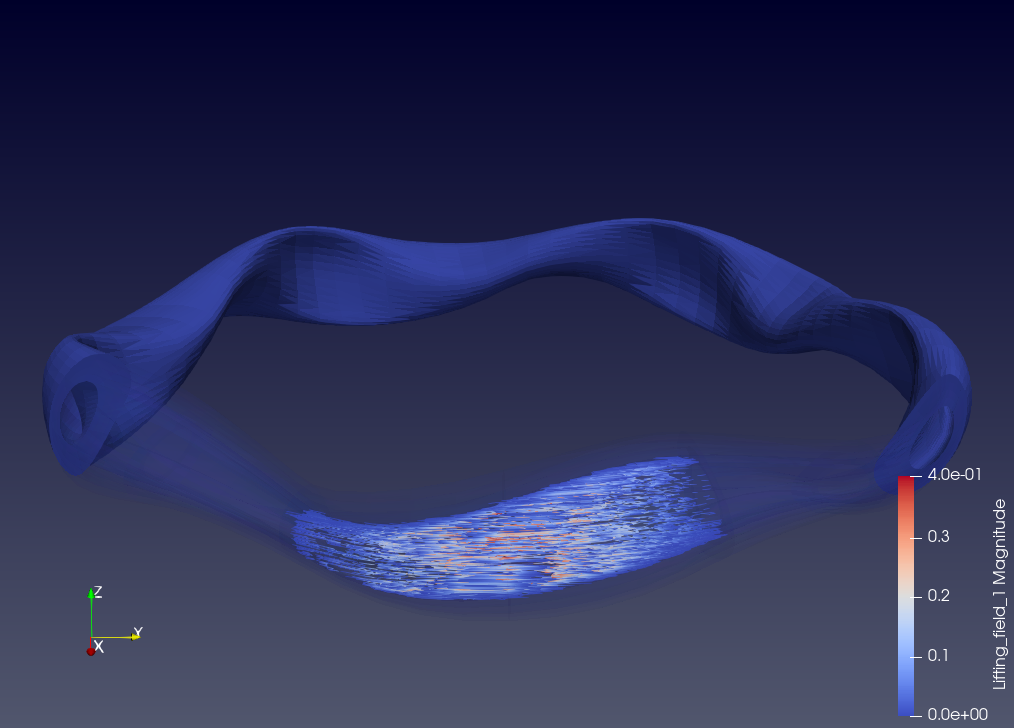}%
    \quad
    \includegraphics[width=.47\textwidth]{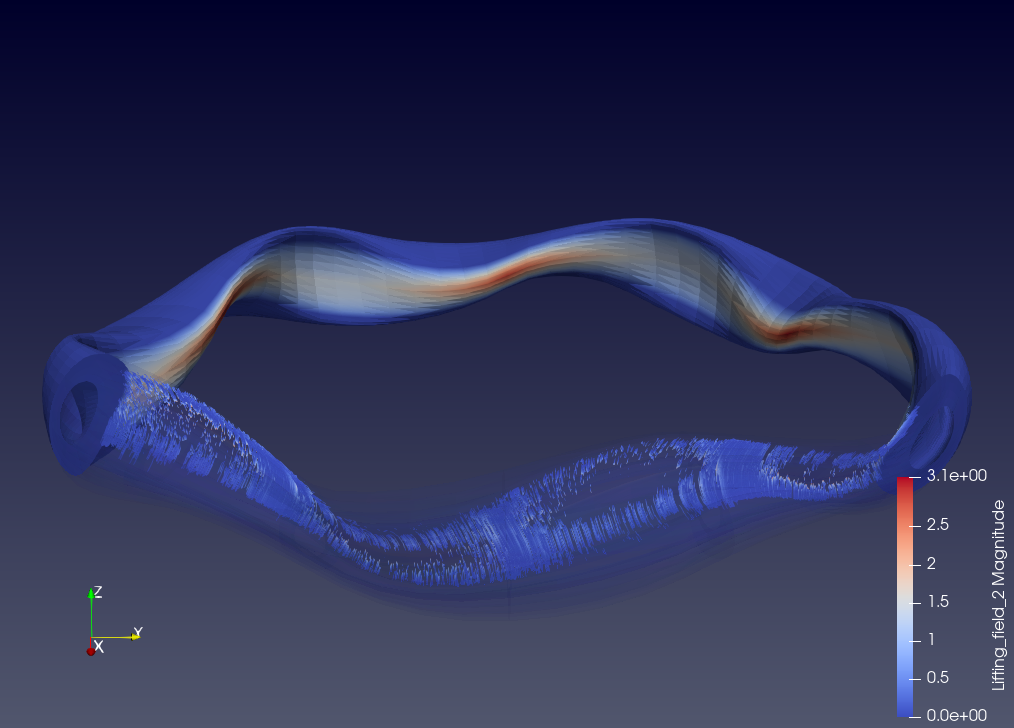}%
    \caption{Poloidal (left) and toroidal (right) lifted boundary potentials $\bA^b_{h,i}$. The initial points for the field line tracing in this and the following Figures have been sampled quasi uniformly. Fluctuations in the field line density bear no meaning.}%
    \label{fig:lifting_fields}%
\end{figure}
Having access to the lifted boundary potentials corresponding to both tunnels, solving the respective discrete Hodge-Laplace problem \eqref{A0h_sys} was realized in PSYDAC. 
\begin{figure}[ht!]%
    \centering
    \includegraphics[width=.47\textwidth]{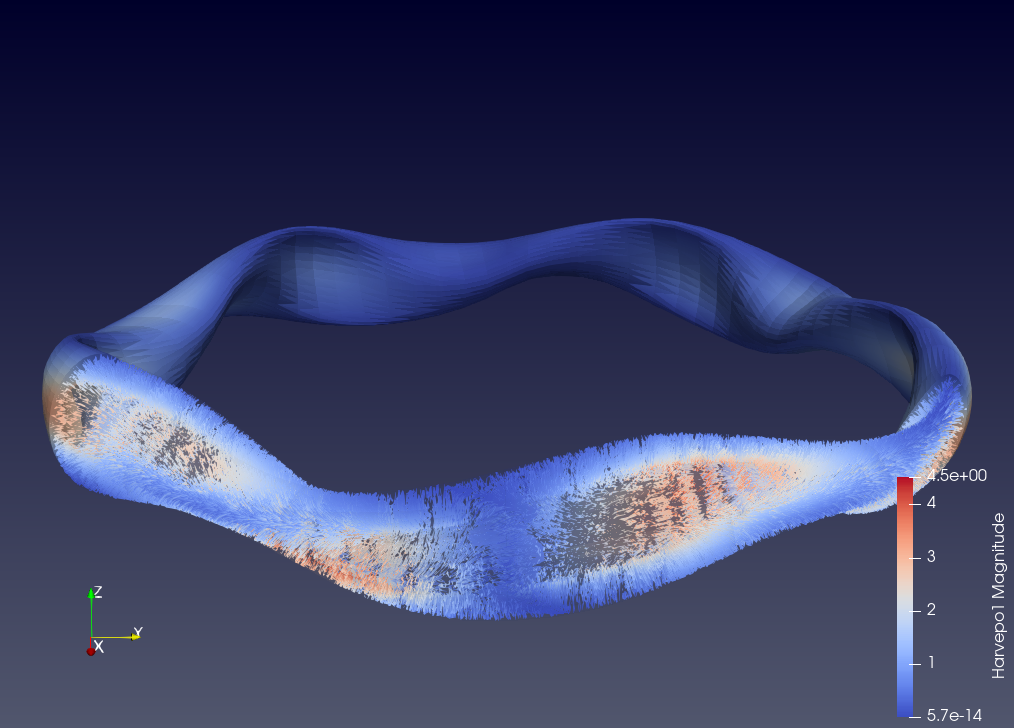}%
    \quad
    \includegraphics[width=.47\textwidth]{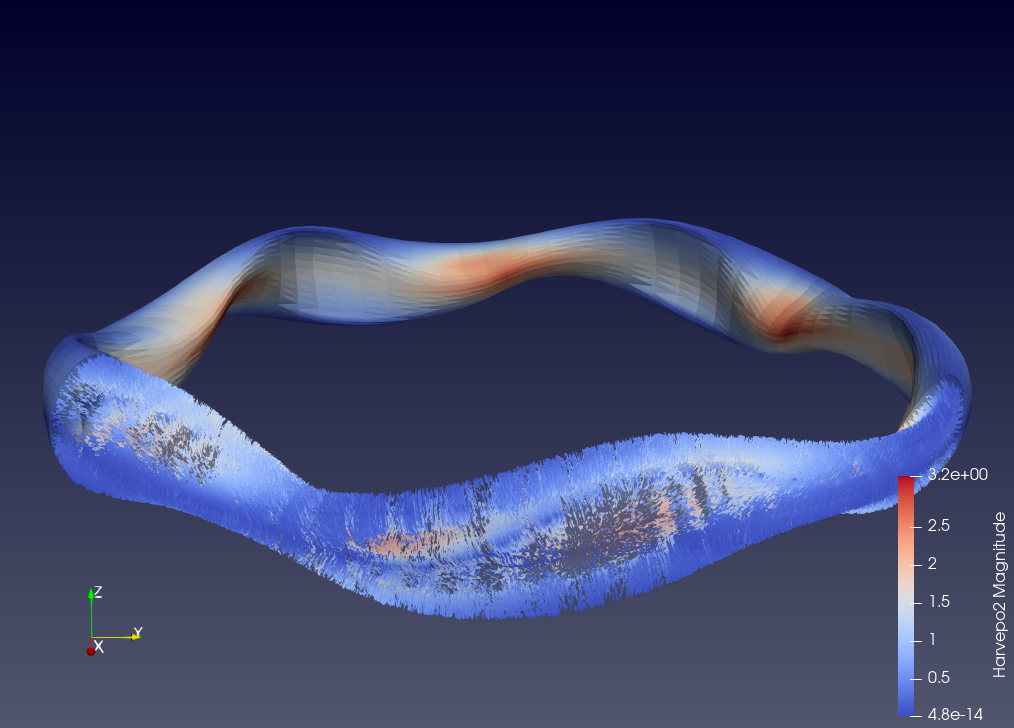}%
    \caption{The corresponding harmonic vector potentials $\bA_{h,i}$.}%
    \label{fig:harvepos}%
\end{figure}
The sum of the lifted boundary potentials $\bA^b_{h,i}$ and the corresponding correction potentials $\bA^0_{h,i}$ provides valid vector potentials $\bA_{h,i}$ for each of the tangent harmonic fields, see Figure \ref{fig:harvepos}. Their tangential traces coincide with those of the respective lifted boundary potentials. 
\begin{figure}[ht!]%
    \centering
    \includegraphics[width=.47\textwidth]{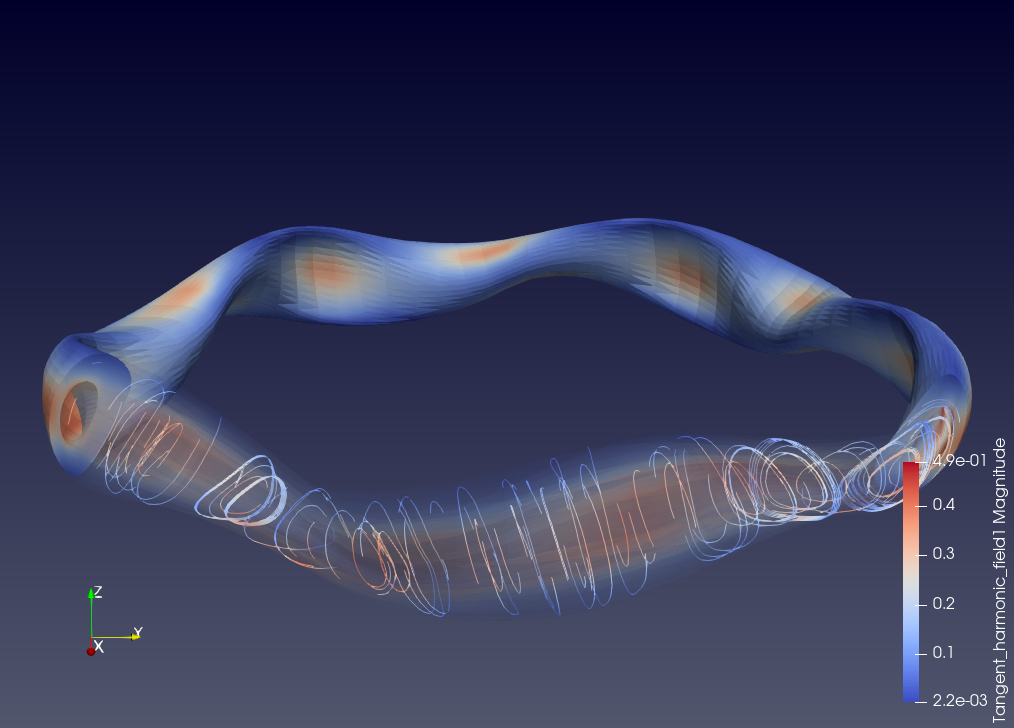}%
    \quad
    \includegraphics[width=.47\textwidth]{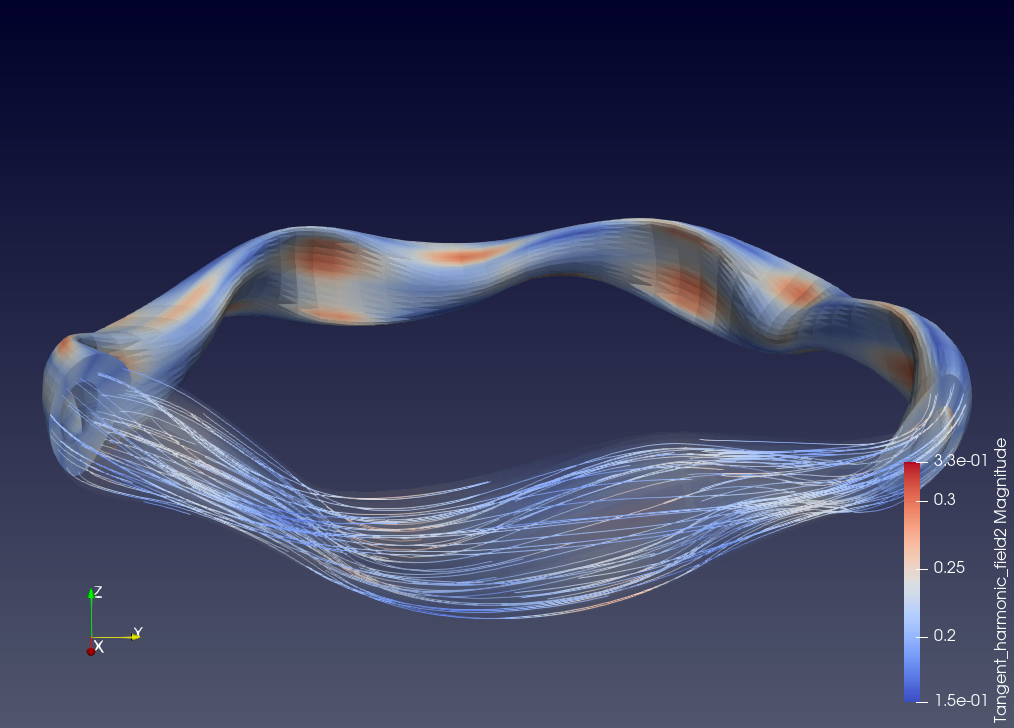}%
    \caption{Poloidal (left) and toroidal (right) tangent harmonic fields $\bw_{h,i} = \bcurl \bA_{h,i}$.}%
    \label{fig:tangent_harmonic_fields}%
\end{figure}

Finally, we compute the curl of these potentials to verify that the resulting vector fields $\bw_{h,i}=\bcurl\bA_{h,i}$ are indeed linearly independent discrete harmonic fields tangent to the boundary. In particular, we observe on Figure \ref{fig:tangent_harmonic_fields} that the lifted boundary potential associated with one (arbitrary) poloidal tunnel curve generates the poloidal tangent harmonic field, and the lifted boundary potential associated with a toroidal tunnel curve generates the toroidal tangent harmonic field.
We also note that computing these tangent harmonic fields on such a parametrized domain is straightforward, as they are the 2-form pushforward of the tangent harmonic fields on a logical hollow cylinder, for which we know the analytical expressions. The benefit of our approach is to be able to compute their vector potentials, and to apply to general domains.

\subsection{Solving a simpler discrete curl-curl problem}

To conclude our study and answer a question raised during a discussion with our colleague Florian Hindenlang, let us examine a simplified version of problem~\eqref{A0h_sys}, where the discrete harmonic space $\tilde \frH^1_{h,0}$ is replaced by a space of the form
\begin{equation*}
    \bgrad \tilde V^0_{h,\psi} :=  \bgrad \Big(\bigoplus_{i = 1}^{\beta_2} \Span(\psi_{h,i})\Big)
\end{equation*}
associated with \textit{arbitrary} functions $\psi_{h,i} \in \tilde V^0_h$ satisfying $\psi_{h,i}|_{S_j} = \delta_{i,j}$ for $0 \le j \le \beta_2$. 
Note that the existence of such functions follows from Property (P3). Thus, the new problem reads:

Find $(\sigma^\psi_{h,i},\bA^{0,\psi}_{h,i},\bp^\psi_{h,i})\in \tilde V^0_{h,0} \times \tilde V^1_{h,0} \times \bgrad \tilde V^0_{h,\psi}$, such that
\begin{equation}
    \label{A0h_sys_psi}
    \left\{\begin{aligned}
        \sprod{\sigma^\psi_{h,i}}{ \tau} + \sprod{\bgrad\tau}{ \bA_{h,i}^{0,\psi}} &= 0,\ 
        &&\quad\forall\ \tau \in \tilde V^0_{h,0}, 
        \\
        \sprod{\bgrad\sigma^\psi_{h,i}}{ \bv} + \sprod{\bcurl\bA_{h,i}^{0,\psi}}{ \bcurl\bv}
        + \sprod{\bp^\psi_{h,i}}{ \bv} &= -\sprod{\bcurl\bA_{h,i}^b}{ \bcurl\bv},\ 
        &&\quad\forall\ \bv\in \tilde V^1_{h,0}, 
        \\
        \sprod{\bA_{h,i}^{0,\psi}}{\bq} &= 0,\ &&\quad\forall\ \bq\in \bgrad \tilde V^0_{h,\psi}.
    \end{aligned}\right.
\end{equation}
This problem may be of practical interest since it does not require to compute the discrete harmonic fields from $\tilde \frH^1_{h,0}$: the harmonic potentials from \eqref{phi0h_eq} may be replaced by other functions that satisfy the same piecewise constant boundary conditions but are not subject to specific constraints inside $\Omega$. In particular one may choose functions with localized supports close to the surfaces $S_i$. 
Given this freedom, there is little hope that problem~\eqref{A0h_sys_psi} satisfies any stability property as the mesh is refined. Nevertheless, 
the following result holds.

\begin{proposition} Problem~\eqref{A0h_sys_psi} admits a unique solution
    which coincides with that of problem~\eqref{A0h_sys}, in the sense that
    \begin{equation}
    (\sigma^\psi_{h,i}, \bA^{0,\psi}_{h,i}, \bp^\psi_{h,i})
    = (\sigma_{h,i}, \bA^{0}_{h,i}, \bp_{h,i}).        
    \end{equation}
\end{proposition}

\begin{proof}
    We will show that any solution of \eqref{A0h_sys_psi} satisfies 
    $\sigma^\psi_{h,i} = 0$, $\bp^\psi_{h,i} = 0$
    and 
    \begin{equation} \label{props_A0h_psi}
        \bA^{0,\psi}_{h,i} \in (\tilde \frZ^1_{h,0})^\perp,
        \qquad
        \sprod{\bcurl(\bA_{h,i}^{0,\psi} + \bA_{h,i}^b)}{ \bcurl\bv}
        = 0, 
        \quad \forall\ \bv\in \tilde V^1_{h,0}
    \end{equation}
    where we remind that 
    $\tilde \frZ^1_{h,0} = \bgrad \tilde V^0_{h,0} \oplus \tilde \frH^1_{h,0}$
    is the kernel of the curl in $\tilde V^1_{h,0}$,
    see \eqref{tilde_Hh}--\eqref{tilde_Zh}.
    Since the relations \eqref{props_A0h_psi} characterize a unique field in $\tilde V^1_{h,0}$, this will show the existence and uniqueness of a solution. The fact that it coincides with the solution of \eqref{A0h_sys} will follow from the fact that the latter satisfies the same relations, 
    as seen in the proof of Theorem~\ref{th:H1h_basis} and Remark~\ref{rem:A0h_perpZ}.

    So, let $\psi_h \in \tilde V^0_{h,\psi}$ be such that 
    $\bp^\psi_{h,i} = \bgrad \psi_h$, and after observing that
    $\bgrad \psi_h \in \tilde V^1_{h,0}$,
    take $\bv = \bgrad \sigma^\psi_{h,i} + \bp^\psi_{h,i} = \bgrad(\sigma^\psi_{h,i} + \psi_h)$
    in \eqref{A0h_sys_psi}: since $\bv$ is curl-free,
    this yields $\bgrad(\sigma^\psi_{h,i} + \psi_h) = 0$,
    hence $\sigma^\psi_{h,i} + \psi_h = 0$ 
    given that this function must vanish on $S_0$.
    It then follows from the direct sum \eqref{tV0hc}
    that $\sigma^\psi_{h,i} = \psi_h = 0$.
    In particular $\bp^\psi_{h,i} = 0$ and the second relation in 
    \eqref{props_A0h_psi} readily follows. To show the first one, 
    observe that the first and last equation from \eqref{A0h_sys_psi}
    now imply 
    \begin{equation*}
        \bA^{0,\psi}_{h,i} \in (\bgrad \tilde V^0_{h,0})^\perp \cap 
        (\bgrad \tilde V^0_{h,\psi})^\perp
        = (\bgrad \tilde V^0_{h,S})^\perp 
        \subseteq (\bgrad \tilde V^0_{h,0} \oplus \tilde \frH^1_{h,0})^\perp
        = (\tilde \frZ^1_{h,0})^\perp
    \end{equation*}
    where the inclusion follows from the fact that 
    $\tilde \frH^1_{h,0}$ (and $\bgrad \tilde V^0_{h,0}$, obviously) 
    is included in $\bgrad \tilde V^0_{h,S}$
    according to Theorem~\ref{th:H2h_basis}.
    This shows the claimed equalities, and ends the proof. 
\end{proof}

\section*{Acknowledgments}

The authors would like to thank Douglas Arnold and Martin Costabel for fruitful exchanges on the de Rham theorem and its extension to Hilbert spaces. Inspiring discussions with several colleagues at IPP and outside are also acknowledged, in particular with Phil Morrison, Dirk Pauly, Patrick Ciarlet, Tobias Blickhan and Ralf Hiptmair.

\bibliographystyle{plain} 
\bibliography{harpo}

\end{document}

%% file: figs/domain_w_curves.pdf_tex
\begingroup%
  \makeatletter%
  \providecommand\color[2][]{%
    \errmessage{(Inkscape) Color is used for the text in Inkscape, but the package 'color.sty' is not loaded}%
    \renewcommand\color[2][]{}%
  }%
  \providecommand\transparent[1]{%
    \errmessage{(Inkscape) Transparency is used (non-zero) for the text in Inkscape, but the package 'transparent.sty' is not loaded}%
    \renewcommand\transparent[1]{}%
  }%
  \providecommand\rotatebox[2]{#2}%
  \newcommand*\fsize{\dimexpr\f@size pt\relax}%
  \newcommand*\lineheight[1]{\fontsize{\fsize}{#1\fsize}\selectfont}%
  \ifx\svgwidth\undefined%
    \setlength{\unitlength}{595.27559055bp}%
    \ifx\svgscale\undefined%
      \relax%
    \else%
      \setlength{\unitlength}{\unitlength * \real{\svgscale}}%
    \fi%
  \else%
    \setlength{\unitlength}{\svgwidth}%
  \fi%
  \global\let\svgwidth\undefined%
  \global\let\svgscale\undefined%
  \makeatother%
  \begin{picture}(1,.8)%
    \lineheight{1}%
    \setlength\tabcolsep{0pt}%
    \put(0,0){\includegraphics[width=\unitlength,page=1]{figs/domain_w_curves.pdf}}%
    \put(.08,.08){\makebox(0,0)[lt]{\lineheight{1.25}\smash{\begin{tabular}[t]{l}\Large$\Omega$\end{tabular}}}}%
    \put(.09,.45){\makebox(0,0)[lt]{\lineheight{1.25}\smash{\begin{tabular}[t]{l}\Large$\color{orange}S_0\color{black}$\end{tabular}}}}%
    \put(.36,.09){\makebox(0,0)[lt]{\lineheight{1.25}\smash{\begin{tabular}[t]{l}\Large$\color{orange}S_1\color{black}$\end{tabular}}}}%
    \put(.68,.13){\makebox(0,0)[lt]{\lineheight{1.25}\smash{\begin{tabular}[t]{l}\Large$\color{orange}S_2\color{black}$\end{tabular}}}}%
    \put(.34,.28){\makebox(0,0)[lt]{\lineheight{1.25}\smash{\begin{tabular}[t]{l}\Large$\color{cyan}\Gamma_{1}\color{black}$\end{tabular}}}}%
    \put(.82,.36){\makebox(0,0)[lt]{\lineheight{1.25}\smash{\begin{tabular}[t]{l}\Large$\color{cyan}\Gamma_{2}\color{black}$\end{tabular}}}}%
    \put(.66,.31){\makebox(0,0)[lt]{\lineheight{1.25}\smash{\begin{tabular}[t]{l}\Large$\color{red}\tilde \Gamma_{1,1}\color{black}$\end{tabular}}}}%
    \put(.63,.4){\makebox(0,0)[lt]{\lineheight{1.25}\smash{\begin{tabular}[t]{l}\Large$\color{red}\tilde \Gamma_{1,2}\color{black}$\end{tabular}}}}%
    \put(.8,.23){\makebox(0,0)[lt]{\lineheight{1.25}\smash{\begin{tabular}[t]{l}\Large$\color{red}\tilde \Gamma_{2,1}\color{black}$\end{tabular}}}}%
    \put(.92,.3){\makebox(0,0)[lt]{\lineheight{1.25}\smash{\begin{tabular}[t]{l}\Large$\color{red}\tilde \Gamma_{2,2}\color{black}$\end{tabular}}}}%
  \end{picture}
\endgroup%

%% file: inkscape_figs/mapping.pdf_tex
\begingroup%
  \makeatletter%
  \providecommand\color[2][]{%
    \errmessage{(Inkscape) Color is used for the text in Inkscape, but the package 'color.sty' is not loaded}%
    \renewcommand\color[2][]{}%
  }%
  \providecommand\transparent[1]{%
    \errmessage{(Inkscape) Transparency is used (non-zero) for the text in Inkscape, but the package 'transparent.sty' is not loaded}%
    \renewcommand\transparent[1]{}%
  }%
  \newcommand*\fsize{\dimexpr\f@size pt\relax}%
  \newcommand*\lineheight[1]{\fontsize{\fsize}{#1\fsize}\selectfont}%
  \ifx\svgwidth\undefined%
    \setlength{\unitlength}{595.27559055bp}%
    \ifx\svgscale\undefined%
      \relax%
    \else%
      \setlength{\unitlength}{\unitlength * \real{\svgscale}}%
    \fi%
  \else%
    \setlength{\unitlength}{\svgwidth}%
  \fi%
  \global\let\svgwidth\undefined%
  \global\let\svgscale\undefined%
  \makeatother%
  \begin{picture}(1,1)%
    \setlength\tabcolsep{0pt}%
    \put(0,0){\includegraphics[width=\unitlength]{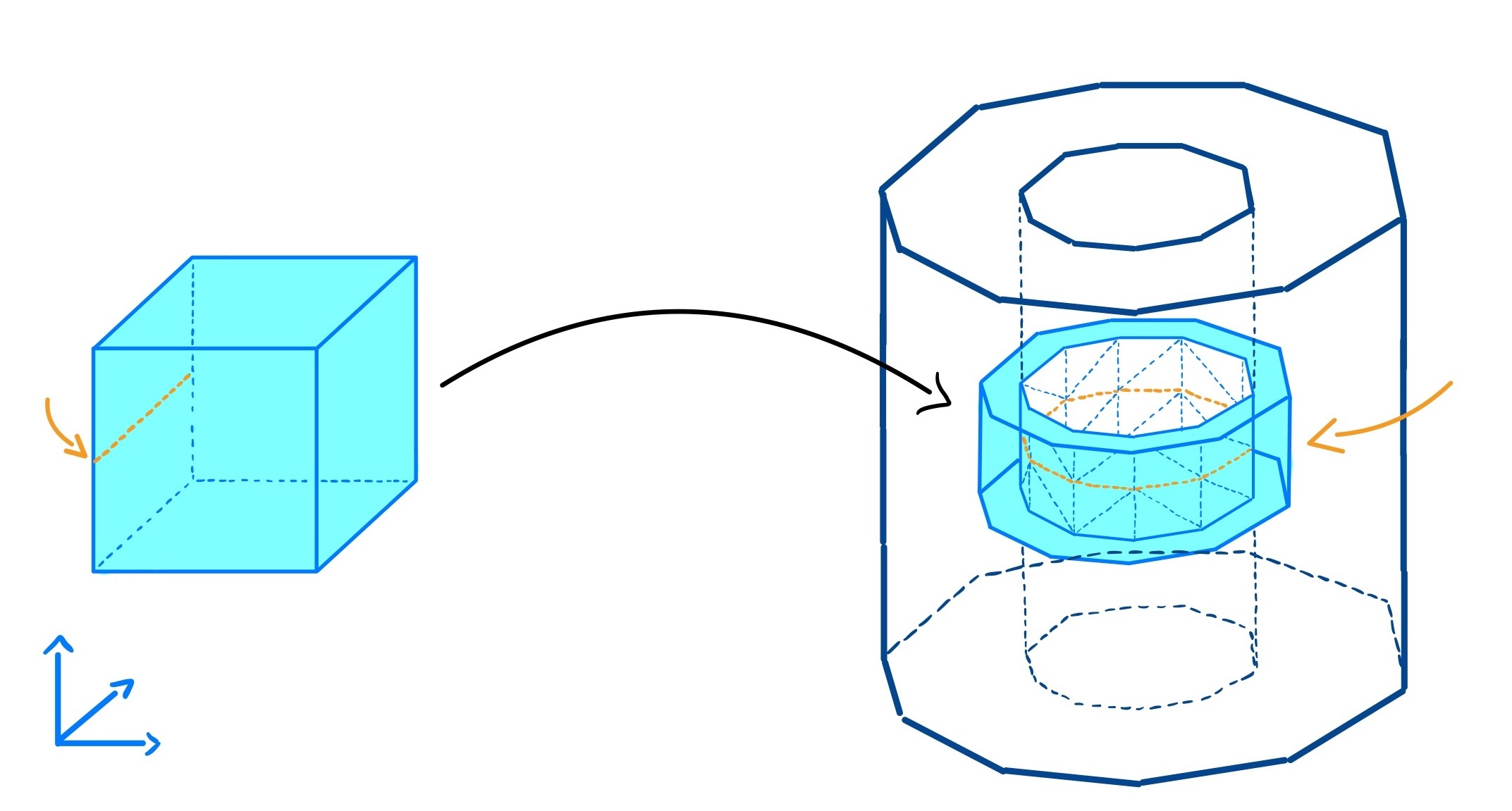}}%
    \put(-.07,.45){\makebox(0,0)[lt]{\lineheight{1.25}\smash{\begin{tabular}[t]{l}\Large$\cO=[0,1)\times\TT\times(-1,1)$\end{tabular}}}}%
    \put(.02,.3){\makebox(0,0)[lt]{\lineheight{1.25}\smash{\begin{tabular}[t]{l}\Large$\color{orange}\hat{\Gamma}\color{black}$\end{tabular}}}}%
    \put(.96,.31){\makebox(0,0)[lt]{\lineheight{1.25}\smash{\begin{tabular}[t]{l}\Large$\color{orange}\Gamma_i\color{black}$\end{tabular}}}}%
    \put(.87,.49){\makebox(0,0)[lt]{\lineheight{1.25}\smash{\begin{tabular}[t]{l}\Large$\Omega$\end{tabular}}}}%
    \put(.33,.37){\makebox(0,0)[lt]{\lineheight{1.25}\smash{\begin{tabular}[t]{l}\Large$T_i:\cO\to X_i$\end{tabular}}}}%
    \put(.055,.12){\makebox(0,0)[lt]{\lineheight{1.25}\smash{\begin{tabular}[t]{l}\Large$\color{black}z\color{black}$\end{tabular}}}}%
    \put(.1,.09){\makebox(0,0)[lt]{\lineheight{1.25}\smash{\begin{tabular}[t]{l}\Large$\color{black}\theta\color{black}$\end{tabular}}}}%
    \put(.12,.03){\makebox(0,0)[lt]{\lineheight{1.25}\smash{\begin{tabular}[t]{l}\Large$\color{black}r\color{black}$\end{tabular}}}}%
  \end{picture}
\endgroup%

%% file: figs/collar.pdf_tex
\begingroup%
  \makeatletter%
  \providecommand\color[2][]{%
    \errmessage{(Inkscape) Color is used for the text in Inkscape, but the package 'color.sty' is not loaded}%
    \renewcommand\color[2][]{}%
  }%
  \providecommand\transparent[1]{%
    \errmessage{(Inkscape) Transparency is used (non-zero) for the text in Inkscape, but the package 'transparent.sty' is not loaded}%
    \renewcommand\transparent[1]{}%
  }%
  \providecommand\rotatebox[2]{#2}%
  \newcommand*\fsize{\dimexpr\f@size pt\relax}%
  \newcommand*\lineheight[1]{\fontsize{\fsize}{#1\fsize}\selectfont}%
  \ifx\svgwidth\undefined%
    \setlength{\unitlength}{203.81911469bp}%
    \ifx\svgscale\undefined%
      \relax%
    \else%
      \setlength{\unitlength}{\unitlength * \real{\svgscale}}%
    \fi%
  \else%
    \setlength{\unitlength}{\svgwidth}%
  \fi%
  \global\let\svgwidth\undefined%
  \global\let\svgscale\undefined%
  \makeatother%
  \begin{picture}(1,0.8504407)%
    \lineheight{1}%
    \setlength\tabcolsep{0pt}%
    \put(0,0){\includegraphics[width=\unitlength,page=1]{collar.pdf}}%
    \put(0.23556509,0.65187667){\color[rgb]{0.82352941,0,0}\makebox(0,0)[rt]{\lineheight{1.25}\smash{\begin{tabular}[t]{r}$\partial \Sigma_i$\end{tabular}}}}%
    \put(0.82793307,0.44139887){\color[rgb]{0,0.47843137,1}\makebox(0,0)[lt]{\lineheight{1.25}\smash{\begin{tabular}[t]{l}$\Gamma_i$\end{tabular}}}}%
  \end{picture}%
\endgroup%